\documentclass{sl2ams}

\usepackage{amscd,amssymb}
\usepackage{amsfonts}
\usepackage{amsmath}
\usepackage[margin=1in]{geometry}

\usepackage{mathtools}

\usepackage{tikz-cd}

\makeatletter
\def\qed@warning{}
\makeatother

\numberwithin{equation}{section}

\usepackage[export]{adjustbox}
\usepackage{graphicx,subcaption}
\graphicspath{{fig/}}

\usepackage{tensor}
\usepackage{enumitem}

\newlist{thmenum}{enumerate}{1}
\setlist[thmenum]{label=(\alph*)}

\usepackage{scalerel}

\newcommand{\nr}{N}
\newcommand{\set}[1]{\left\{ \def\given{\ \middle| \ }  #1 \right\}  }
\newcommand{\defeq}{\mathrel{:=}}

\def\leftfun#1{\mathopen{}\left#1}
\def\rightfun#1{\right#1}

\newcommand{\iso}{\cong}
\DeclareMathOperator{\tr}{tr}
\DeclareMathOperator{\id}{id}

\DeclareMathOperator{\End}{End}
\DeclareMathOperator{\spec}{Spec}

\renewcommand{\Im}{\operatorname{Im}}

\DeclareMathOperator{\Div}{Div}

\NewDocumentCommand{\Wchar}{}{\mathcal{X}}
\NewDocumentCommand{\Wcharhat}{}{\widehat{\mathcal{X}}}

\NewDocumentCommand{\slg}{}{\operatorname{SL}_2(\mathbb{C})}

\NewDocumentCommand{\sla}{}{\mathfrak{sl}_2}

\NewDocumentCommand{\pf}{O{} m }{\varphi_{#1}\leftfun(#2\rightfun)}
\NewDocumentCommand{\pfname}{O{} }{\varphi_{#1}}
\NewDocumentCommand{\qlf}{O{} m m}{\Lambda_{#1}\leftfun(#2 \middle| #3 \rightfun)}
\NewDocumentCommand{\qlfname}{}{\Lambda}
\NewDocumentCommand{\df}{O{} m}{\mathrm{D}\leftfun(#2 \rightfun)}
\newcommand{\dil}{\operatorname{Li}_2}
\newcommand{\dilr}{\mathcal{L}}

\NewDocumentCommand{\modb}{m }{\left\{#1\right\}_{\nr}}
\NewDocumentCommand{\cutoff}{m }{\theta_{\nr}\leftfun(#1\rightfun)}
\NewDocumentCommand{\nrdel}{m m}{\overline{\delta}_{#1}^{#2}}
\NewDocumentCommand{\del}{m m}{\delta_{#1}^{#2}}

\NewDocumentCommand{\qlog}{m m}{w\leftfun(#1 \middle | #2 \rightfun)}
\NewDocumentCommand{\qlogname}{}{w}
\NewDocumentCommand{\qp}{m O{\omega} m}{\left(#1 ; #2\right)_{#3}}
\NewDocumentCommand{\fstack}{m m m}{f\leftfun( \begin{matrix} #1 \\ #2 \end{matrix}\middle| #3 \rightfun)}
\NewDocumentCommand{\fstacksmall}{m m m}{f\leftfun( \begin{smallmatrix} #1 \\ #2 \end{smallmatrix}\middle| #3 \rightfun)}

\newcommand{\RIII}{\operatorname{R3}}

\NewDocumentCommand{\rep}{m}{V\leftfun(#1\rightfun)}
\NewDocumentCommand{\BasisMacro}{m o m}{
  \IfNoValueTF{#2}{%
    {#1}_{#3}
  }{%
    {#1}(#2)_{#3}
  }
}
\NewDocumentCommand{\vb}{m}{v_{#1}}
\NewDocumentCommand{\vbh}{m}{\widehat{v}_{#1}}

\NewDocumentCommand{\rmat}{o m m m m}{
  \IfNoValueTF{#1}{%
    \widehat{R}_{#2 #3}^{#4 #5}
  }{%
    \widehat{R}(#1)_{#2 #3}^{#4 #5}
  }
}
\NewDocumentCommand{\rmatm}{m m m m}{\widehat{\overline{R}}_{#1 #2}^{#3 #4}}

\NewDocumentCommand{\jfunc}{O{\xi}}{\mathcal{J}_{#1}}
\NewDocumentCommand{\jinv}{O{\xi}}{\operatorname{J}_{#1}}

\NewDocumentCommand{\lN}{}{\mathrm{r}}
\NewDocumentCommand{\lW}{}{\mathrm{t}}
\NewDocumentCommand{\lS}{}{\mathrm{l}}
\NewDocumentCommand{\lE}{}{\mathrm{b}}

\declaretheorem[style=theorem]{proposition}
\declaretheorem[style=theorem,sibling=proposition]{theorem}
\declaretheorem[style=theorem,sibling=proposition]{lemma}

\declaretheorem[style=definition,sibling=proposition]{definition}
\declaretheorem[style=definition,sibling=proposition]{remark}

\usepackage[noabbrev,capitalize]{cleveref}
\crefname{equation}{equation}{equations}
\Crefname{equation}{Equation}{Equations}
\crefname{bigtheorem}{Theorem}{Theorems}
\Crefname{bigtheorem}{Theorem}{Theorems}

\newcommand{\CC}{\mathbb{C}}
\newcommand{\ZZ}{\mathbb{Z}}

\newcommand{\J}{\mathcal{J}}

\newcommand{\Z}{\mathcal{Z}}
\newcommand{\U}{\mathcal{U}}
\newcommand{\W}{\mathcal{W}}
\newcommand{\R}{\mathcal{R}}

\newcommand\mybar{\kern1pt\rule[-\dp\strutbox]{1pt}{\baselineskip}\kern1pt}

\title[The holonomy braiding for \(\U_\xi(\mathfrak{sl}_2)\)]{The holonomy braiding for \(\U_\xi(\mathfrak{sl}_2)\) in terms of geometric quantum dilogarithms}
\author{Calvin McPhail-Snyder}
\address{C.M-S.: Department of Mathematics, Duke University}
\email{calvin@sl2.site}
\author{Nicolai Reshetikhin}
\address{N.R.: Department of Mathematics, Yau Center for Mathematical Sciences, Tsinghua University, Beijing; BIMSA, Beijing; St. Petersburg State University, St. Petersburg; Department of Mathematics, UC Berkeley}
\email{reshetik@math.berkeley.edu}


\begin{document}

\begin{abstract}
  We derive an explicit formula for the holonomy $R$-matrix of quantum $\mathfrak{sl}_2$ at a root of unity.
  We show it factorizes into a product of four quantum dilogarithms and satisfies a holonomy Yang-Baxter equation.
  This factorization extends previously known results and we collect many existing results needed for our computation.
\end{abstract}

\maketitle

\tableofcontents

\section{Introduction}
\label{sec:Introduction}

Holonomy $R$-matrices were used in \cite{Kashaev2005} to describe invariants of knots with flat connections in the complement.
These $R$-matrices were introduced in \cite{Reshetikhin1995} to describe the braiding for the de Concini-Kac specialization of quantum groups corresponding to simple finite dimensional Lie algebras at a root of unity. 
In that case the quantized universal enveloping algebra is finite dimensional over the central Hopf subalgebra.
Generic irreducible modules are parametrized by generic points of a finite cover of the dual Poisson Lie group with the standard Poisson-Lie structure \cite{DeConcini1990}.
In this paper we focus on the case of $\mathfrak{sl}_2$.

It is known that when the universal \(R\)-matrix is evaluated in representations obtained through a certain homomorphism from quantum
$\mathfrak{sl}_2$ to the \(q\)-Weyl algebra it factorizes into the product of
four quantum dilogarithms, see for example \cite{Fadeev2000,Schrader2019} and references therein. 
This factorization in the setting of the modular double was first observed 
in \cite{Fadeev2000}. In \cite{Kashaev1996} this factorization was was derived as a general consequence of homomorphisms from the Drinfeld double to the Heisenberg double of a Hopf algebra. 

In \cite{Faddeev1994, Bazhanov1995} a geometric version of the quantum dilogarithm 
was established. The geometric \(q\)-dilogarithm is the regular part of the asymptotic of \(q\)-dilogarithm when $q$ goes to a root of unity as in \cite{Bazhanov1995}. It was shown to satisfy the holonomy pentagon equation. 

In this paper we describe the factorization of the holonomy $R$-matrix for quantum $\mathfrak{sl}_2$ \cite{Reshetikhin1995}
into four holonomy quantum dilogarithms. We construct this factorization ``from scratch'', directly from the representation theory of the de Concini-Kac version of quantum $\mathfrak{sl}_2$ at a root of unity.
Our \(R\)-matrices are sections of a vector bundle over the space of geometric parameters and we obtain local coordinates by taking certain logarithms.
We show that our \(R\)-matrices satisfy a holonomy Yang-Baxter relation when these log-parameters satisfy a natural geometric condition.

In \cite{McPhailSnyderVolume} our \(R\)-matrices are used to define knot and tangle invariants which are a refinement of those in \cite{Blanchet2018}.
We expect that  using the factorization that we describe here one can clearly establish the relation 
between invariants of knots with flat connection in the complement constructed in \cite{Kashaev2005} and 
the work \cite{Baseilhac2004} where similar invariants were obtained using triangulations of the complement.
We also expect that our factorization is an important step in constructing the corresponding homotopy TQFT \cite{Turaev2010}; see also \cite{GeerBook}.

This paper is a mix of an overview of known material and original results.
We include the overview to collect material scattered over a large number of papers.

In \cref{sec:q-algebra} we establish conventions on quantum \(\sla\) and its presentation in terms of Weyl algebras. \Cref{sec:root of unity} is an overview of elements of representation theory of quantum \(\sla\) and its braiding.
In \cref{sec:R-matrix} we construct the braiding matrix (the holonomy $R$-matrix) using $q$-dilogarithms and prove it satisfies holonomy braid relations.
\Cref{sec:R-matrix-computations}  contains some technical computations involving the \(R\)-matrix that we use elsewhere in the paper.
\cref{sec:quantum-dilogarithms} has technical details on quantum dilogarithms and collects properties used throughout the paper.

\subsubsection*{Acknowledgements}
We would like to thank Nathan Geer, Rinat Kashaev, and Bertrand Patureau-Mirand for helpful discussions. 
The work of N.~R.\ was supported by the Collaboration Grant ``Categorical Symmetries'' from
the Simons Foundation, by the Changjiang fund, and by the project 075-15-2024-631 funded
by the Ministry of Science and Higher Education of the Russian Federation.

\section{Quantum \texorpdfstring{$\sla$}{sl2} and Weyl algebras}
\label{sec:q-algebra}

Here we recall basic facts about quantum \(\sla\), its structure at a root of unity and its braiding properties.

\subsection{The braiding for quantum \texorpdfstring{$\sla$}{sl2}}
\label{sec:qgrp-conventions}

Recall that $\U_q(\sla)$ is the $\CC[q, q^{-1}]$-algebra generated by $K^{\pm 1}, E, F$ with defining relations%
\note{
  Our normalization is different from the standard one \cite{Drinfeld1987,Lusztig1993}.
  It represents the integral form of quantum \(\mathfrak{sl}_{2}\) which specializes to a root of unity as in \cite{DeConcini1990}.
}
\begin{align*}
  KE &= q^2 EK \\
  KF &= q^{-2} FK \\
  [E,F] &= (q - q^{-1}) (K - K^{-1}).
\end{align*}
It is a Hopf algebra, with the coproduct
\[
\Delta(K)=K\otimes K, \ \ \Delta(E)=E\otimes K+1\otimes E, \ \ \Delta(F)=F\otimes 1+ K^{-1}\otimes F
\]
and antipode
\[
  S(K) = K^{-1}, \, \, S(E) = - EK^{-1}, \, \, S(F) = -KF.
\]
The center of $\U_q(\sla)$ is freely generated by the Casimir element\note{The algebra 
$\U_q(\sla)$ defined above, including the Casimir element is actually defined over $\ZZ[q,q^{-1}]$ but we will not use this property here.}
\begin{equation}
  \label{eq:Casimir}
  \Omega = EF + q^{-1} K + q K^{-1}.
\end{equation}

\subsection{The braiding for \texorpdfstring{$\U_q(\sla)$}{Uq(sl2)}}
\label{sec:outer-R-matrix}
The algebra $\U_q(\sla)$ is not quasitriangular.
Instead it is \emph{braided} \cite{Reshetikhin1995} which means there is an outer automorphism 
\[
 \R :  \Div(\U_q(\sla)^{\otimes 2}) \to \Div(\U_q(\sla)^{\otimes 2})
 \]
of the division ring of $\U_q(\sla)^{\otimes 2}$ which intertwines the coproduct and opposite coproduct
\begin{equation*}
  \R(\Delta(x)) =  \Delta^{op}(x)
\end{equation*}
and satisfies hexagon identities
\begin{align*}
  (\Delta \otimes 1) \R(x \otimes y) &= \R_{13} \R_{23} (\Delta(x) \otimes y) \\
  (1 \otimes \Delta) \R(x \otimes y) &= \R_{13} \R_{12} (x \otimes \Delta(y))
\end{align*}
that imply the Yang-Baxter equation.
Here $\R_{ij}$ means that $\R$ acts on tensor factors $i$ and $j$.
That is, for each \(x, y \in \U_q\) one can write
\[
  \R(x \otimes y) = \sum_k f_k \otimes g_k
\]
for some \(f_k, g_k \in \U_q\) and we define
\[
  \R_{13}(x \otimes z \otimes y) = \sum_k f_k \otimes z \otimes g_k
\]
and so on.%
\note{
  Here the summations are finite, but the terms are elements of the corresponding division algebras.
}
The automorphism \(\R\) is derived from the conjugation action of the universal \(R\)-matrix in the \(h\)-adic quantum group, \(q = e^h\) \cite{Kashaev2004}.

$\R$ acts on generators as
\begin{align*}
  \R(\Delta(x)) &= \Delta^{\operatorname{op}}(x), \ x \in \U_\xi \\
  \R(K_1) &= K_1 (1  - q^{-1} K_1^{-1} E_1 F_2 K_2 ) \\
  \R(E_1) &= E_1 K_2 \\
  \R(F_2) &= K_1^{-1} F_2
\end{align*}
where we abbreviate $K_1 = K \otimes 1, E_2 = 1 \otimes E$, etc.
\begin{proposition}[\cite{Kashaev2004}]
  $\R$ acts on the remaining generators of $\U_q^{\otimes 2}$ by:
  \begin{align*}
    \R(K_2) &= (1 - q^{-1} K_1^{-1} E_1 F_2 K_2)^{-1} K_2 \\
    \R(E_2) &= E_1 + K_1 E_2 - E_1 K_2^2 (1 - q K_1^{-1} E_1 F_2 K_2)^{-1} \\
    \R(F_1) &= F_1 K_2^{-1} + F_2 - K_1^{-2} F_2(1 - q K_1^{-1} E_1 F_2 K_2)^{-1}
  \end{align*}
  and preserves Casimirs:
  \(
    \R(\Omega_1) = \Omega_1, \R(\Omega_2) = \Omega_2.
  \)
\end{proposition}
These follow from computations like
\[
  \R(1 \otimes K) = \R(K \otimes 1)^{-1} \R(K \otimes K) = \R(K \otimes 1)^{-1} K \otimes K
\]
and
\[
  \R(1 \otimes E) = \R(E \otimes K + 1 \otimes E - E \otimes K) = E \otimes 1 + K \otimes E - \R(E \otimes K).
\]

\subsection{Representation into the Weyl algebra}

The $q$-Weyl algebra $W_q$ is an associative algebra over $\CC[q, q^{-1}]$ generated by invertible generators $x,y$ satisfying
  \[
    xy = q^2 yx.
  \]
The following is well known:
\begin{proposition}
  There is a unique algebra homomorphism $\phi : \U_q(\sla) \to W_q[z,z^{-1}]$ acting on generators as 
  \[
    \phi(K)=x, \ \ \phi(E)=qy(z-x), \ \ \phi(F)=y^{-1}(1-z^{-1}x^{-1})
  \]
\end{proposition}
It is easy to compute the image of the Casimir:
\[
  \phi(\Omega) = qz + (qz)^{-1}.
\]
From now on we write \(\W_q = W_q[z, z^{-1}]\).
\begin{proposition}
  \label{thm:R-action-on-weyl}
  There exists a unique algebra automorphism $\R^\W$ of the division algebra $\Div{\W_q^{\otimes 2}}$ of $\W_q^{\otimes 2}$ which acts on the generators as 
  \begin{align*}
    \R^\W(z_1)&=z_1 \\
    \R^\W(z_2)&=z_2 \\
    \R^\W(x_1)&=x_1g, \\
    \R^\W(x_2)&=g^{-1}x_2, \\
    \R^\W(y_1^{-1})&=y_2^{-1}+(y_1^{-1}-z_2^{-1}y_2^{-1})x_2^{-1}, \\
    \R^\W(y_2)&=\frac{z_1}{z_2}y_1+(y_2-{z_2}^{-1}y_1)x_1, \\
    \intertext{
    where
    \(
    x_1 = x \otimes 1, y_2 = 1 \otimes y, z_1 = z \otimes 1, \text{ etc.\ and} 
    \)
  }
    g&=1-x_1^{-1}y_1(z_1-x_1)y_2^{-1}(x_2-z_2^{-1}).
  \end{align*}
  The following diagram commutes:
  \begin{equation}
    \label{eq:automorphism-pullback}
    \begin{tikzcd}
     \Div(\U_q^{\otimes 2}) \arrow[r, "\R"] \arrow[d, "\phi \otimes \phi"] & \Div(\U_q^{\otimes 2}) \arrow[d, "\phi \otimes \phi"] \\
      \Div(\W_q^{\otimes 2}) \arrow[r, "\R^W"] & \Div(\W_q^{\otimes 2})
    \end{tikzcd}
  \end{equation}
  The inverse of $\R^W$ acts on generators as
  \begin{align*}
    (\R^W)^{-1}(x_1)&=x_1\widetilde{g}^{-1}, \\
    (\R^W)^{-1}(x_2)&=\widetilde{g}x_2, \\
    (\R^W)^{-1}(y_1^{-1})&=\frac{z_1}{z_2}y_2^{-1} + (y_1^{-1}-{z_1}y_2^{-1})x_2,\\
    (\R^W)^{-1}(y_2)&=y_1 + (y_2 - {z_1}y_1)x_1^{-1},
  \end{align*}
  where 
  \[
    \widetilde{g}= 1 - y_1 (z_1 - x_1) y_2^{-1} (1 - z_2^{-1} x_2^{-1}).\qedhere
  \]
\end{proposition}
\begin{proof}
  The rules for $x_1$ and $x_2$ follow directly from those for $K_1$ and $K_2$ and it is natural to choose \(\R^\W\) to preserve the \(z_i\) as \(\R\) preserves the Casimirs.
  We give $y_1$ as an example, and $y_2$ can be computed similarly.
  Since $\R(E_1) = E_1 K_2$, we have $\R^W(q y_1(z_1 - x_1)) = q y_1(z_1-x_1) x_2$, or
  \[
    \R^W(y_1)
    \left[
      z_1 - x_1 + y_1(z_1 - x_1) y_2^{-1} (x_2 - z_2^{-1})
    \right]
    =
    y_1(z_1 - x_1)x_2.
  \]
  We can multiply by  $y_1^{-1}(z_1 - x_1)^{-1}$ on the right to get
  \[
    \R^W(y_1) \left[ y_1^{-1} + y_2^{-1} (x_2 - z_2^{-1}) \right] = x_2
  \]
  from which it is easy to derive
  \[
    \R^W(y_1^{-1}) = y_2^{-1}+(y_1^{-1}-z_2^{-1}y_2^{-1})x_2^{-1}.\qedhere
  \]
\end{proof}

\section{Specialization to a root of unity}
\label{sec:root of unity}

\subsection{Central subalgebras at roots of unity}

Set $\xi=\exp(\frac{\pi i}{N})$ for \(N \ge 2\) an integer.%
\note{
  Our results will still work for \(\xi = \exp(\pi i m /N)\) for \(m, N\) relatively prime, but we take \(m = 1\) for simplicity.
}
Denote by $\U_\xi$ the specialization of $\U_q(\sla)$ to $q=\xi$.
\begin{proposition}[\cite{DeConcini1990}]
  \label{thm:qgrp-center}
  \begin{thmenum}
    \item
    The center \(\Z\) of $\U_\xi$ is generated by
    \[
      \Z_0 = \CC[K^{\pm N}, E^N, F^N]
    \]
    and the Casimir element $\Omega$ \cref{eq:Casimir}, subject to the relation
    \[
      P_N(\Omega) = E^N F^N - (K^N + K^{-N}) 
    \]
    where $P_N$ is the $N$th renormalized Chebyshev polynomial defined by the identity
    \[
      P_N(t + t^{-1}) = t^N+ t^{-N}.
    \]
    \item
      The central subalgebra \(\Z_0\) is a Hopf subalgebra of  \(\U_\xi\).
  \end{thmenum}
\end{proposition}
As for any commutative Hopf algebra, the set of characters (algebra homomorphisms) \(\chi : \Z_0 \to \mathbb{C}\) is a group with multiplication
\[
  (\chi_1 \cdot \chi_2)(x) \defeq (\chi_1 \otimes \chi_2)(\Delta(x)).
\]
In this case the Hopf subalgebra $\Z_0$ is isomorphic to the algebra of functions on the algebraic group
\[
  \slg^*
  =
  \set{
    \left(
      \begin{bmatrix}
        \kappa & 0 \\
        \phi & 1
      \end{bmatrix}
      ,
      \begin{bmatrix}
        1 & \epsilon \\
        0 & \kappa
      \end{bmatrix}
    \right) 
    \given
    \kappa \ne 0
  }
  \subseteq
  \operatorname{GL}_2(\mathbb{C})
  \times
  \operatorname{GL}_2(\mathbb{C}).
\]
The algebra \(\U_\xi\) is finite-dimensional over \(\Z_0\), so we can think of \(\U_\xi\) as a sheaf of algebras over \(\operatorname{Spec} \Z_0 = \slg^*\).
The map
\[
  \psi : \slg^* \to \slg, \psi(x^+, x^-) = x^+ (x^-)^{-1}
\]
is a birational equivalence, but not a group homomorphism.

\subsection{\texorpdfstring{$\U_\xi$}{Uξ}-modules.}

We say a \(\U_\xi\) module \(V\) has character \(\chi \in \slg^*\) if
\[
  z \cdot v = \chi(z) v \text{ for all } v \in V, z \in \Z_0.
\]
In particular any simple \(\U_\xi\)-module has a character.
Because the multiplication on characters is compatible with the coproduct the category of \(\U_\xi\)-modules with characters is graded by \(\slg^*\): if \(V_1, V_2\) have characters \(\chi_1, \chi_2\), then \(V_1 \otimes V_2\) has character \(\chi_1 \cdot \chi_2\).

A character \(\hat \chi : \Z \to \CC\) on the full center extending some \(\chi \in \slg^*\) is specified by a solution of the equation
\[
  \hat \chi(P_N(\Omega)) = \chi(E^N F^N - (K^N + K^{-N}))
  =
  - \tr \psi(\chi)
\]
If the eigenvalues of \(\psi(\chi)\) are \(m, m^{-1}\), then the Casimir will satisfy
\[
  \hat \chi(\Omega) = \xi^{2\mu + 1} + \xi^{-(2\mu-1)}
\]
for some \(\mu\) with \(\xi^{2\mu} =\omega^{\mu} = m\).
For the modules used in this paper (\cref{def:standard-rep}) the choice of \(\chi\) and \(\mu\) completely determines the isomorphism class of a simple module, so the choice of \(\mu\) is analogous to the choice of a highest weight.

However, unlike for ordinary \(\U_q\)-modules \(\mu\) may not be a highest weight, or even a weight at all.
Whenever the character of \(V\) has \(\chi(E^N), \chi(F^N) \ne 0\)  the generators  \(E\) and \(F\) act invertibly on \(V\).
We call such modules  \defemph{cyclic}, and in this case \(\xi^{2\mu}\)  will \emph{not} be an eigenvalue of \(K\).
While one can still consider a weight basis for \(V\) by diagonalizing \(K\) there is no canonical way to choose a \emph{highest} weight because \(E\) does not have a kernel.
We will give an explicit description of some cyclic modules in \cref{sec:cyclic-reps}.

\subsection{The braiding on characters}
\label{sec:braiding root of unity}
Next we describe how \(\R\) gives an action of the braid group on \(\Z_{0}\)-characters.
Set \(Y = 1 + K_1^{-N} E_1^N F_2^N K_2^N \in \Z_0 \otimes \Z_0\).
In the ring \(\U_\xi^{\otimes 2}[Y^{-1}]\) the element \((1 - q^{-1} K_1^{-1} E_1 F_2 K_2)\) is invertible, so \(\R\) induces a map
\[
  \R :
  \U_\xi^{\otimes 2}
  \to
  \U_\xi^{\otimes 2}[Y^{-1}]
\]
we continue to denote by \(\R\).
For characters \(\chi_1, \chi_2\) the equation
\[
  (\chi_{1'} \otimes \chi_{2'}) (\R (x)) = (\chi_1 \otimes \chi_2), \quad x \in \U_\xi^{2}
\]
uniquely defines characters \(\chi_{1'}, \chi_{2'}\) whenever \((\chi_1 \otimes \chi_2)(\R^{-1}(Y)) \ne 0\).
(Later we will write out this condition explicitly for central \(\W\)-characters.)
As such, we obtain a partially defined, invertible map
\[
  B : \slg^* \times \slg^*  \to \slg^* \times \slg^*
  ,
  B(\chi_1, \chi_2) = (\chi_{2'}, \chi_{1'})
\]
that satisfies the braid relation
\[
  (B \times \id)
  (\id \times B )
  (B \times \id)
  =
  (\id \times B )
  (B \times \id)
  (\id \times B )
\]
on the subset of \((\slg^*)^{\times 3}\) for which each side is defined.
Here by ``partially defined'' we mean that the domain \(B\) is really a subset  \( (\slg^* \times \slg^*) \setminus \mathfrak{A}\) avoiding the singular pairs \(\mathfrak{A}\) of characters discussed above.
The braiding \(B\) was first studied by \textcite{Weinstein1992}.

The map \(B\) is an example of a generically defined biquandle \cite[Section 6]{Blanchet2018}, which means that it makes sense to color tangle diagrams by elements of \(\slg^*\) related by \(B\) at the crossings.
For example, \(B\) satisfies a braid relation, so when modifying a diagram by a Reidemeister type 3 move the labellings can be compatibly modified.
Geometrically these colorings can be understood as defining a (gauge class of) flat \(\sla\) connections on the braid complement \cite{Kashaev2005,Blanchet2018,McPhailSnyder2022}.

\subsection{The homomorphism from \texorpdfstring{$\U_q$}{Uq} to the Weyl algebra at root of unity}
We now consider the presentation of \(\U_q\) in terms of \(\W_q\) at \(q = \xi\).
The map $\phi$ takes the center of $\U_\xi$ to the center of $\W_\xi$:
\begin{lemma}
  The center of $\W_{\xi}$ is generated by $z$, $x^N$, and $y^N$ and $\phi$ takes the center of $\U_\xi$ to the center of $\W_{\xi}$.
  Explicitly,
  \begin{equation*}
    \begin{aligned}
      \phi(K^N)
      &=
      x^N
      \\
      \phi(E^N)
      &=
      y^N(x^N - z^N)
      \\
      \phi(F^N)
      &=
      y^{-N}(1 - z^{-N} x^{-N})
    \end{aligned} 
  \end{equation*}
\end{lemma}
\begin{proof}
  The claim about the center of \(\W\) is clear.
  $K^N$ is obvious and $F^N$ follows from the same reasoning as $E^N$.
  For $E^N$, notice that
  \begin{align*}
    \phi(E^N)
    &=
    \left( \xi y (z - x) \right)^{N}
    \\
    &=
    (\xi y)^2 (z - \xi^2 x) (z - x) \left( y (z - x ) \right)^{N-2}
    \\
    &\dots
    =
    (\xi y)^N \prod_{k=0}^{N-1} (z - \xi^{2k} x)
  \end{align*}
  All the terms in the product except $z^N$ and $x^N$ vanish.
  The coefficient of $z^N$ is clearly $1$, while the coefficient of $x^N$ is $(-1)^{N}$ times $\xi$ raised to the power
  \(
    \sum_{k=0}^{N-1} (2 k) =  N(N-1)
  \)
  so that
  \[
    \phi(E^N)
    = -y^N( z^N + (-1)^{N} \xi^{N(N-1)} x^N)
    = -y^N( z^N + (-1)^{N} (-1)^{N-1} x^N)
    = y^N( x^N - z^N). \qedhere
  \]
\end{proof}

As before the automorphism \(\R^\W\)  of \(\Div(\W_q^{\otimes 2})\) induces a map on the root-of-unity specialization.
It is easy to compute the action on the generators:

\begin{proposition}
  \label{prop:RW-action-on-center}
  The action of $\R^W$ on the center of $\Div(\W_\xi^{\otimes 2})$ is given by
  \begin{align*}
    \R^W(z_1)&=z_1, \\
    \R^W(z_2)&=z_2, \\
    \R^W(x_1^N)&=x_1^N G , \\
    \R^W(x_2^N)&=x_2^N G^{-1}, \\
    \R^W(y_1^{-N})&= y_2^{-N}+\left(y_1^{-N} - \frac{y_2^{-N}}{z_2^{N}}\right)x_2^{-N}, \\
    \R^W(y_2^N)&=\frac{z_1^N}{z_2^{N}}y_1^{N}+ \left(y_2^N - \frac{y_1^N}{z_2^N}\right)x_1^N,
  \end{align*}
  where
  \begin{equation*}
    G = 1 + x_1^{-N} \frac{y_1^N}{y_2^N}(x_1^N - z_1^N)(x_2^N - z_2^{-N}).
  \end{equation*}
  The inverse action is
  \begin{align*}
    (\R^W)^{-1}(x_1^N)&= x_1^N \widetilde G^{-1}, \\
    (\R^W)^{-1}(x_2^N)&= x_2^N \widetilde G, \\
    (\R^W)^{-1}(y_1^{-N})&= \frac{z_1^N}{z_2^N}y_2^{-N}+(y_1^{-N} - z_1^N y_2^{-N})x_2^{N} ,  \\
    (\R^W)^{-1}(y_2^{N})&=y_1^{N}+(y_2^N - z_1^N y_1^N)x_1^{-N},
  \end{align*}
  where
  \begin{equation*}
    \widetilde{G} = 1 + x_2^{-N} \frac{y_1^N}{y_2^N} (x_1^N - z_1^N)(x_2^N - z_2^{-N}).
    \qedhere
  \end{equation*}
\end{proposition}

\begin{proof}
  These follow from the $q$-binomial theorem and the action on the generators in \cref{thm:R-action-on-weyl}.
\end{proof}

When acting on the algebra itself (not the division algebra) the map \(\R^{W}\) is not always well-defined: one needs to ensure that each invertible generator is sent to a nonzero scalar.
Before we dealt with this by inverting a certain element.
Here it is more convenient to work directly with characters.

\begin{definition}
  The central subalgebra \(\Z_0^\W \defeq \CC[x^{\pm N}, y^{\pm N}, z^{\pm N}] \subseteq \W_\xi\) maps to \(\Z_0\) under \(\phi\).
  We write \(\Wchar\) for the set of characters \(\chi : \Z_0^\W \to \CC\).
  Such a character is determined by the numbers
  \[
    a = \chi(x^N), b = \chi(y^N), m = \chi(z^N).
  \]
  and it corresponds to the \(\Z_0\)-character with
  \begin{gather*}
    \chi(K^{N})  = a
    \\
    \chi(E^{N})  = b(a - m)
    \\
    \chi(F^{N}) = (ab)^{-1}(a - m^{-1}).
  \end{gather*}

  The full center of \(\W_\xi\) is \(\Z^\W = \Z_0^\W[z]\).
  Characters of this algebra are elements of \(\Wchar\) plus a choice of \(N\)th root of \(m = \chi(z^N)\).
  We denote the space of such characters \(\hat \chi\) by \(\Wcharhat\).
  There is an obvious \(N\)-fold covering map \(\Wcharhat \to \Wchar\).
\end{definition}

\begin{definition}
  \label{def:braiding}
  Let \(\chi_1, \chi_2 \in \Wchar\) and write \(\chi_i(x^N) = a_i, \chi_i(y^N) = b_i, \chi_i(z^N) = m_i\).
  We say the pair \((\chi_1, \chi_2)\) is \defemph{admissible} if the complex numbers defined by
  \begin{gather}
    \label{eq:a-transf-positive}
    \begin{aligned}
      a_{1'}
      &=
      a_1 A^{-1}
      \\
      a_{2'}
      &=
      a_2 A
      \\
      A &= 1 - \frac{m_1 b_1}{b_2} \left(1 - \frac{a_1}{m_1}\right)\left(1 - \frac{1}{m_2 a_2}\right)
    \end{aligned}
    \\
    \label{eq:b-transf-positive}
    \begin{aligned}
      b_{1'}
      &=
      \frac{m_2 b_2}{m_1}
      \left(
        1 - m_2 a_2 \left( 1 - \frac{b_2}{m_1 b_1} \right)
      \right)^{-1}
      \\
      b_{2'}
      &=
      b_1
      \left(
        1 - \frac{m_1}{a_1}\left( 1 - \frac{b_2}{m_1 b_1} \right)
      \right)
    \end{aligned}
    \\
    \label{eq:m-transf-positive}
    \begin{aligned}
      m_{1'}
      &= m_1
      &
      m_{2'}
      &= m_2
    \end{aligned}
  \end{gather}
  are not \(0\) or \(\infty\).
  In this case we write \((\chi_{2'} , \chi_{1'}) = B(\chi_1, \chi_2)\) where \(\chi_{1'}(x^N) = a_{1'}\) and so on.
  There is an obvious extension to elements of \(\Wcharhat\) by setting \(\hat \chi_1(z) = \hat \chi_{1'}(z)\) and \(\hat \chi_2(z) = \hat \chi_{2'}(z)\) and admissibility only depends on the image in \(\Wchar\).
\end{definition}

As before \(B\) is a partially defined map on pairs of characters.
The inverse map \(B^{-1}(\chi_1, \chi_2) = (\chi_{2'}, \chi_{1'})\) is given by
\begin{gather}
  \label{eq:a-transf-negative}
  \begin{aligned}
    a_{1'}
      &=
      a_1 \tilde A^{-1}
      \\
      a_{2'}
      &=
      a_2 \tilde A
      \\
      \tilde A
      &=
      1 - \frac{b_2}{m_1 b_1}\left(1 - m_1 a_1 \right)\left(1 - \frac{m_2}{a_2}\right).
  \end{aligned}
  \\
  \label{eq:b-transf-negative}
  \begin{aligned}
    b_{1'}
      &=
      \frac{m_2 b_2}{m_1}
      \left(
        1 - \frac{a_2}{m_2} \left( 1 - \frac{m_1 b_1}{b_2} \right)
      \right)
      \\
      b_{2'}
      &=
      b_1
      \left(
        1 - \frac{1}{m_1 a_1} \left( 1 - \frac{m_1 b_1}{b_2} \right)
      \right)^{-1}
  \end{aligned}
  \\
  \label{eq:m-transf-negative}
  \begin{aligned}
    m_{1'}
      &= m_1
      &
      m_{2'}
      &= m_2
  \end{aligned}
\end{gather}

\begin{proposition}
  \label{thm:braiding formulas}
  Whenever \(B(\chi_1, \chi_2) = (\chi_{2'}, \chi_{1'})\) we have
  \[
    (\hat \chi_1 \otimes \hat \chi_2)(x) = (\hat \chi_{1'} \otimes \hat \chi_{2'})\R^\W(x)
  \]
  for every \(x \in \Z^\W \otimes \Z^\W\).
\end{proposition}
\begin{proof}
  This follows from the action of \(\R^\W\) on the center computed in \cref{prop:RW-action-on-center}.
  It is easiest to use
  \[
    (\hat \chi_1 \otimes \hat \chi_2)\left((\R^W)^{-1}(x)\right) = (\hat \chi_{1'} \otimes \hat \chi_{2'})(x)
  \]
  to compute \(B\).
\end{proof}

\subsection{Braiding of \texorpdfstring{$\U_\xi$}{Uξ}-modules.}
The map \(\mathcal{R}\) gives a braiding on  the algebra \(\U_{\xi}\), but we want to construct a braiding on \(\U_{\xi}\)-modules.
Given \(\U_\xi\)-modules \((\pi_{i}, V_{i})\) with characters \(\chi_{i}\) related as in \cref{thm:braiding formulas} a \defemph{holonomy \(R\)-matrix} is a linear map
\[
  R : V_{1} \otimes V_{2} \to V_{1'} \otimes V_{2'}
\]
intertwining \(\R\) in the sense that the diagram 
\begin{equation}
  \label{eq:R-mat-U-module-diagram}
  \begin{tikzcd}[row sep = large, column sep = large]
    \U_\xi^{\otimes 2} \arrow[r, "\R"] \arrow[d, swap, "\pi_{1} \otimes \pi_{2}"] & \U_\xi^{\otimes 2}[Y^{-1}] \arrow[d, "\pi_{1}' \otimes \pi_{2}'"] \\
    \End_\CC(V_1 \otimes V_2) \arrow[r, "a \mapsto R a R^{-1}"] & \End_\CC(V_1' \otimes V_2')
  \end{tikzcd}
\end{equation}
commutes.
Concretely, this means that
\[
  (\pi_{1'} \otimes \pi_{2'})( \R(u) ) R = R  (\pi_{1} \otimes \pi_{2})(u)
\]
for every \(u \in \U_{\xi} \otimes \U_{\xi}\).
Here we emphasize that \(R\) depends not only on \(V_{1}\) and \(V_{2}\) but also \(V_{1'}\) and \(V_{2'}\) as in general these will not be isomorphic to \(V_{1}\) and \(V_{2}\).
This is in contrast to modules over a quasitriangular Hopf algebra, where the universal \(R\)-matrix canonically defines a braiding for any pair of modules.
As such, a holonomy braiding is a choice of both a continuous family of modules \(\rep{\hat{\chi}}\) for each character \(\hat{\chi}\) and also of braiding maps \(R\) between them satisfying a parametrized Yang-Baxter equation.
In more geometric language we seek to construct a bundle of \(\U_{\xi}\)-modules over \(\spec \Z\) (an \(\nr\)-fold cover of \(\spec \Z_{0} = \slg^*\)) and a family of braiding maps between them.

Recall that if \(V\) is a \(\W_{\xi}\)-module it becomes a \(\U_{\xi}\)-module via
\[
  u \cdot x = \phi(u) \cdot x \text{ for } x \in V, u \in \U_{\xi}.
\]
Under this correspondence a character \(\hat \chi \in \Wcharhat\) with
\begin{align*}
  \hat{\chi}(x^{\nr}) &= a
                      &
  \hat{\chi}(y^{\nr}) &= b
                      &
  \hat{\chi}(z) &= \omega^{\mu}
\end{align*}
corresponds to the \(\Z\)-character
\begin{align*}
  \hat{\chi}(\phi(K^{\nr}))  &= a
                             &
  \hat{\chi}(\phi(E^{\nr}))  &= b(a - m)
  \\
  \hat{\chi}(\phi(F^{\nr})) &= (ab)^{-1}(a - m^{-1})
                            &
  \hat{\chi}(\phi(\Omega)) &= \omega^{\mu + 1/2} + \omega^{-(\mu + 1/2)}
\end{align*}
hence to the group element
\[
  \left(
      \begin{bmatrix}
        a & b(a-m) \\
        0 & 1
      \end{bmatrix}
      ,
      \begin{bmatrix}
        1 & 0 \\
        b^{-1}(a - m^{-1}) & a
      \end{bmatrix}
  \right)
      \in \slg^*
\]
whose image in \(\slg\) is
\begin{equation}
  \label{eq:defactorization image}
  \psi(\chi) = 
  \begin{bmatrix}
    a & -b(a-m) \\
    b^{-1}(a-m^{-1}) & m + m^{-1} -a
  \end{bmatrix}
  \in \slg
\end{equation}
The map \(\psi : \Wchar \to \slg\) defined in \cref{eq:defactorization image} is not surjective, but it is surjective up to gauge equivalence: for every \(g \in \slg\) there is an \(h\) so that \(\psi(\chi) = h^{-1} g h\).

Thus for every pair \((\hat \chi_{1}, \hat \chi_{2})\) of admissible \(\W_{\xi}\)-characters we seek to define modules \(\rep{\hat \chi_{i}}\) and linear maps \(R = R(\hat \chi_{1}, \hat \chi_{2})\) so that each diagram 
\begin{equation}
  \label{eq:R-mat-module-diagram}
  \begin{tikzcd}[row sep = large, column sep = large]
    \W_\xi / I_{\hat \chi_{1}}
    \otimes
    \W_\xi / I_{\hat \chi_{2}}
    \arrow[r, "\R^{\W}"] \arrow[d, swap, "\pi_{\hat \chi_1} \otimes \pi_{\hat \chi_2}"] &
    \W_\xi / I_{\hat \chi_{1'}}
    \otimes
    \W_\xi / I_{\hat \chi_{2'}}
    \arrow[d, "\pi_{\hat \chi_{1'}} \otimes \pi_{\hat \chi_{2'}}"] \\
    \End(V(\hat \chi_{1}) \otimes V(\hat \chi_{2})) \arrow[r, "a \mapsto R a R^{-1}"] & \End(V(\hat \chi_{1'}) \otimes V(\hat \chi_{2'}))
  \end{tikzcd}
\end{equation}
commutes, where $\pi_{\hat \chi_i} : \W_\xi \to \End(V(\hat \chi_i))$ is the structure map of the representation $V(\hat \chi_i)$ and \(I_{\hat \chi}\) is the ideal generated by the kernel of the central character.
We can use $\phi$ to pull back the solution of (\ref{eq:R-mat-module-diagram}) to a solution of (\ref{eq:R-mat-U-module-diagram}), as the latter is the composition of the diagrams (\ref{eq:automorphism-pullback}) and (\ref{eq:R-mat-module-diagram}) after taking appropriate specializations.

The modules \(\rep{\hat \chi}\) form a nontrivial vector bundle over \(\spec \Wcharhat\); to describe the matrices \(R\) we need to choose concrete \(\W_{\xi}\)-module structures on them, i.e.\ we must locally trivialize.
The holonomy \(R\)-matrices will depend on this choice of trivialization, but we will show that it has a geometric interpretation that allows one to define an appropriate Yang-Baxter equation.
We further explore the vector bundle perspective in \cite{McPhailSnyderVolume}.

\subsection{\texorpdfstring{\(\W_\xi\)}{Wξ}-modules}
\label{sec:cyclic-reps}

Here we define a local trivialization of the bundle of irreducible representations of \(\W_{\xi}\).

\begin{definition}
  \label{def:standard-rep}
  Recall that \(\omega = \xi^2 = \exp(2\pi i /N)\)  and set \(\omega^{x} =\exp(2 \pi i x/N)\) for all $x\in \CC$.
  For any $\alpha, \beta, \mu \in \CC$, let $\rep{\alpha, \beta, \mu} = (\pi_{\alpha, \beta, \mu}, \CC^N)$ be the \(\W_\xi\)-module defined by
  \begin{equation}
    \label{eq:weight-basis}
    \pi_{\alpha, \beta, \mu}(x) v_n = \omega^{\alpha - n} v_n,
    \quad
    \pi_{\alpha, \beta, \mu}(y) v_n = \omega^{\beta} v_{n-1},
    \quad
    \pi_{\alpha, \beta, \mu}(z) v_n = \omega^{\mu} v_n
  \end{equation}
  with indices considered modulo \(N\).
  Since \(x = \phi(K)\) and \(y\) act diagonalizably on \(\rep{\alpha, \beta, \mu}\) we call it a \defemph{weight module}.
\end{definition}
\(\rep{\alpha, \beta, \mu}\) has character \(\hat \chi \in \Wcharhat\) defined by 
\[
  \chi(x^{\nr}) = \omega^{\nr \alpha}
  \ \ \pi(y^\nr) = \omega^{\nr \beta}
  \ \
  \pi(z) = \omega^{\mu}
\]
and it is clear that for any integers  \(k_i\),
\[
  \rep{\alpha + k_1, \beta + k_2, \mu + \nr k_3}
  \simeq
  \rep{\alpha , \beta , \mu }
\]
so the isomorphism type of \(\rep{\alpha, \beta, \mu}\) is completely determined by \(\hat \chi\).
The parameters \(\alpha, \beta, \mu\) are local coordinates on the bundle of modules.

\begin{theorem}
  \label{thm:R-matrix-exists}
  For any admissible characters \((\hat \chi_1, \hat \chi_2)\) an invertible linear operator 
  \[
    R :
    \rep{\hat{\chi}_{1}} \otimes \rep{\hat{\chi}_{2}}
    \to
    \rep{\hat{\chi}_{1'}} \otimes \rep{\hat{\chi}_{2'}}
  \]
  satisfying (\ref{eq:R-mat-module-diagram}) exists and is unique up to an overall scalar.
\end{theorem}
\begin{proof}

  For each \(i\) the module $V(\hat \chi_i)$ is an irreducible \(\W_\xi\)-module,%
  \note{
    It may be reducible as a \(\U_\xi\)-module if \(\psi(\chi_i)\) is a central element of \(\slg\).
  }
  so its endomorphism algebra is isomorphic to the algebra
  \[
    W_\xi/I_{\hat \chi_i}.
  \]
  We see that $\R$ induces an automorphism
  \[
    \W_\xi/I_{\hat \chi_1} \otimes \W_\xi/I_{\hat \chi_2} \to
    \W_\xi/I_{\hat \chi_1'} \otimes \W_\xi/I_{\hat \chi_2'}
  \]
  hence an automorphism
  \[
    \End( (\CC^N)^{\otimes 2}) \to \End( (\CC^N)^{\otimes 2})
  \]
  of matrix algebras.
  Any such automorphism is inner and given by conjugation by some invertible matrix $R$, unique up to an overall scalar.
\end{proof}

\section{The \texorpdfstring{\(\R\)}{R}-matrix in the Fourier dual basis}
\label{sec:R-matrix}

\subsection{Recursions for the \texorpdfstring{\(\R\)}{R}-matrix coefficients}
\label{subsec:recurrences}

\begin{marginfigure}
\begingroup%
  \makeatletter%
  \providecommand\color[2][]{%
    \errmessage{(Inkscape) Color is used for the text in Inkscape, but the package 'color.sty' is not loaded}%
    \renewcommand\color[2][]{}%
  }%
  \providecommand\transparent[1]{%
    \errmessage{(Inkscape) Transparency is used (non-zero) for the text in Inkscape, but the package 'transparent.sty' is not loaded}%
    \renewcommand\transparent[1]{}%
  }%
  \providecommand\rotatebox[2]{#2}%
  \newcommand*\fsize{\dimexpr\f@size pt\relax}%
  \newcommand*\lineheight[1]{\fontsize{\fsize}{#1\fsize}\selectfont}%
  \ifx\svgwidth\undefined%
    \setlength{\unitlength}{102.79326153bp}%
    \ifx\svgscale\undefined%
      \relax%
    \else%
      \setlength{\unitlength}{\unitlength * \real{\svgscale}}%
    \fi%
  \else%
    \setlength{\unitlength}{\svgwidth}%
  \fi%
  \global\let\svgwidth\undefined%
  \global\let\svgscale\undefined%
  \makeatother%
  \begin{picture}(1,1.29135843)%
    \lineheight{1}%
    \setlength\tabcolsep{0pt}%
    \put(0,0){\includegraphics[width=\unitlength,page=1]{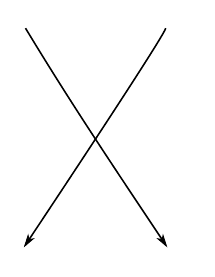}}%
    \put(0.04113266,1.19662491){\makebox(0,0)[lt]{\lineheight{1.25}\smash{\begin{tabular}[t]{l}$2$\end{tabular}}}}%
    \put(0.0411314,0.02664476){\makebox(0,0)[lt]{\lineheight{1.25}\smash{\begin{tabular}[t]{l}$1'$\end{tabular}}}}%
    \put(0.75477955,1.19663464){\makebox(0,0)[lt]{\lineheight{1.25}\smash{\begin{tabular}[t]{l}$1$\end{tabular}}}}%
    \put(0.75477661,0.01715746){\makebox(0,0)[lt]{\lineheight{1.25}\smash{\begin{tabular}[t]{l}$2'$\end{tabular}}}}%
    \put(0.42720925,0.85286315){\makebox(0,0)[lt]{\lineheight{1.25}\smash{\begin{tabular}[t]{l}$\lW$\end{tabular}}}}%
    \put(0.42716249,0.40597543){\makebox(0,0)[lt]{\lineheight{1.25}\smash{\begin{tabular}[t]{l}$\lE$\end{tabular}}}}%
    \put(0.63408158,0.61104761){\makebox(0,0)[lt]{\lineheight{1.25}\smash{\begin{tabular}[t]{l}$\lN$\end{tabular}}}}%
    \put(0.22946709,0.61103793){\makebox(0,0)[lt]{\lineheight{1.25}\smash{\begin{tabular}[t]{l}$\lS$\end{tabular}}}}%
  \end{picture}%
\endgroup%

  \caption{Labels for the segments and the regions near a crossing.}
  \label{fig:crossing-regions}
\end{marginfigure}

We can now explicitly compute the matrix coefficients of \(R\) as a function of the geometric parameters \(\hat \chi_{i}\).
We begin by fixing some terminology and conventions.

\begin{definition}
  We say a crossing is \defemph{\(\chi\)-colored} if it its segments (labeled as in \cref{fig:crossing-regions}) are assigned characters \(\hat \chi_i\) satisfying
  \begin{equation*}
    (\hat \chi_{2'}, \hat \chi_{1}') =
    \begin{cases}
      B(\hat \chi_{1}, \hat \chi_{2}) & \text{for a positive crossing}
      \\
      B^{-1}(\hat \chi_{1}, \hat \chi_{2}) & \text{for a negative crossing}
    \end{cases}
  \end{equation*}
\end{definition}

\begin{definition}
  A \defemph{log-coloring} of a \(\chi\)-colored crossing is a choice of complex numbers \(\beta_{i}\) for the segments, \(\gamma_{i}\) for the regions, and \(\mu_{1}, \mu_{2}\) for the components so that
  \[
    \begin{gathered}
      \hat \chi_{i}(x^{\nr}) = \omega^{\nr \alpha_i}
      \\
      \hat \chi_{i}(y^{\nr}) = \omega^{\nr \beta_i}
      \\
      \hat \chi_{i}(z) = \omega^{ \mu_i}
    \end{gathered}
  \]
  for each segment, where the segment parameters \(\alpha_{i}\) are determined by the region parameters as
  \begin{align*}
    \alpha_1 &= \gamma_{\lW}-\gamma_{\lN}
             &
    \alpha_{2'} &= \gamma_{\lE} - \gamma_{\lN}
    \\
    \alpha_2 &= \gamma_{\lS}- \gamma_{\lW}
             &
    \alpha_{1'} &= \gamma_{\lS} - \gamma_{\lE}
  \end{align*}
  and the labels are as in \cref{fig:crossing-regions}.
\end{definition}

Because \(a_1 a_2 = a_{1'} a_{2'}\) it is natural to require \(\alpha_1 + \alpha_2 = \alpha_{1'} + \alpha_{2'}\), and this is equivalent to a choice of region parameters \(\gamma_j\) as above.
Just as we associate an \(R\)-matrix to a crossing we assign the parametrized \(R\)-matrices of \cref{thm:R-matrix-exists} to \(\chi\)-colored crossings.
Later in \cref{sec:braid groupoids} we will show our \(R\)-matrices satisfy a version of the Yang-Baxter equation, subject to a geometric condition on the log-colorings.

The choice of log-coloring fixes identifications \(\rep{\hat \chi_{i}} \iso \rep{\alpha_{i}, \beta_{i}, \mu_{i}}\) with the weight modules of \cref{def:standard-rep}.
It turns out to be easier to work in the Fourier dual basis \(\set{\vbh{n}}\) of \(\rep{\alpha, \beta, \mu}\) defined by 
\begin{equation}
  \label{eq:non-weight-basis}
  \pi_{\alpha, \beta, \mu}(x) \vbh n = \omega^{\alpha} \vbh{n-1},
  \quad
  \pi_{\alpha, \beta, \mu}(y) \vbh n = \omega^{\beta + n} \vbh n,
  \quad
  \pi_{\alpha, \beta, \mu}(z) \vbh n = \omega^{\mu} \vbh n.
\end{equation}
It is related to the weight basis by
\begin{equation}
  \label{eq:basis-and-dual-basis}
  \vbh n \defeq \sum^{\nr-1}_{k=0} \omega^{nk} \vb k
  \text{ and }
  \vb n = \frac{1}{\nr} \sum_{k=0}^{\nr-1} \omega^{-n k} \vbh k
  .
\end{equation}
Our goal is to determine the matrix coefficients \(\rmat{n_1}{ n_2}{n_1'}{n_2'}\) defined by 
\begin{equation}
  \label{eq:R-matrix-coefficients}
  R \cdot \vbh{n_1 n_2} = \sum_{n_1' n_2'} \rmat{n_1}{ n_2}{n_1'}{n_2'} \vbh{n_1' n_2'}
\end{equation}
where we abbreviate $\vbh{n_1 n_2} = \vbh{n_1} \otimes \vbh{n_2}$.
Here and below all sums are over $\ZZ/\nr\ZZ$.

Assume the crossing is positive.
Abbreviate the structure maps of the representations at the crossing as
\[
  \pi = \pi_1 \otimes \pi_2, \ \ \pi' = \pi_{1'} \otimes \pi_{2'}.
\]
Commutativity of the diagram in \eqref{eq:R-mat-module-diagram} requires matrix equations
\begin{align}
  \label{eq:intertwiner-rel}
  R \pi(u) &=  \pi'\left(\R(u)\right) R, \ \ u \in \W_\xi, \\
  \intertext{and}
  \label{eq:intertwiner-rel-inv}
  \pi'(u) R &= R \pi\left(\R^{-1}(u)\right), \ \ u \in \W_\xi.
\end{align}

Setting $u = y_1^{-1}$ in \eqref{eq:intertwiner-rel} gives the relation
\begin{equation}
  \label{eq:yinv-rel}
  \begin{aligned}
    R \pi(y_1^{-1}) &= \pi'\left( y_2^{-1} + (y_1^{-1} - z_2^{-1} y_2^{-1})x_2^{-1}\right) R \\
                    &= \pi'\left( y_2^{-1} + x_2^{-1}(y_1^{-1} - \omega^{-1} z_2^{-1} y_2^{-1})\right) R
  \end{aligned}
\end{equation}
In terms of the matrix coefficients \eqref{eq:R-matrix-coefficients} this becomes
\begin{align*}
  &  \sum_{n_1' n_2'} \rmat{n_1}{n_2}{n_1'}{n_2'}  \omega^{-\beta_1 - n_1}  \vbh{n_1' n_2'} \\
  &= \sum_{n_1' n_2'} \rmat{n_1}{ n_2}{n_1'}{ n_2'} \omega^{-\beta_{2'} - n_2'} \vbh{n_1' n_2'} \\
  &\phantom{=} + \sum_{n_1' n_2'} \rmat{n_1}{ n_2}{n_1'}{ n_2'}\left[ \omega^{-\beta_{1'} - n_1'}  - \omega^{-\mu_2} \omega^{-\beta_{2'} - n_2' - 1} \right] \omega^{-\alpha_{2'}} \vbh{n_1', n_2'+1}
\end{align*}
which we can rewrite as the recursion
\[
  \rmat{n_1}{n_2}{n_1'}{n_2'} \left[ \omega^{-\beta_1 - n_1}  - \omega^{-\beta_{2'} - n_2'} \right]
  =
  \rmat{n_1}{n_2}{n_1',}{n_2'-1} \omega^{-\alpha_{2'}} \left[ \omega^{-\beta_{1'} - n_1'} - \omega^{- \mu_2 - \beta_{2'} - n_2'} \right]
  .
\]
In a slightly more convenient form, this is
\begin{align*}
  \rmat{n_1}{n_2}{n_1'}{n_2' }
  &=
  \rmat{n_1}{n_2}{n_1',}{n_2' -1}
  \omega^{-\alpha_{2'} - \mu_2}
  \frac
  {1 - \omega^{\mu_2 + \beta_{2'} - \beta_{1'} - 1} \omega^{n_2' - n_1'}}
  {1 - \omega^{\beta_{2'} - \beta_{1}} \omega^{n_2' - n_1}} 
  .
\end{align*}
Abbreviate
\begin{align*}
  \zeta_{\lN}^0
  &=
  \beta_{2'} - \beta_1
  \\
  \zeta_{\lW}^0
  &=
  \beta_{2} - \beta_{1} - \mu_1 
  \\
  \zeta_{\lS}^0
  &=
  \beta_{2} - \beta_{1'} + \mu_2 - \mu_1 
  \\
  \zeta_{\lE}^0
  &=
  \beta_{2'} - \beta_{1'} + \mu_2 
\end{align*}
If any \(\zeta_i^0 \in \ZZ\) we say that the (coloring of the) crossing is \defemph{pinched}; this condition depends only on the characters \(\chi_i\), not the logarithms, and has a geometric interpretation (\cref{rem:pinched eigenlines}).
We can obtain the \(R\)-matrix at pinched crossings as a limit of the general case (see \cref{sec:pinched-limit}), so we exclude them for now.

\begin{theorem}
  \label{thm:R-recurrences}
  At a positive crossing that is not pinched the matrix coefficients satisfy recurrence relations
  \begin{align}
    \label{eq:R-recurrence-i}
    \rmat{n_1}{n_2}{n_1'}{n_2' }
    &=
    \rmat{n_1}{n_2}{n_1',}{n_2' -1}
    {\omega^{-\alpha_{2'} - \mu_2} }
    \frac
    {1 - \omega^{\zeta_E^0 + n_2' - n_1'}}
    {1 - \omega^{\zeta_N^0 +  n_2' - n_1}} 
    \\
    \label{eq:R-recurrence-ii}
    \rmat{n_1}{n_2}{n_1',}{n_2'}
    &=
    \rmat{n_1}{n_2}{n_1'-1,}{n_2'}
    \omega^{-\alpha_{1'} + \mu_1}
    \frac
    {1 - \omega^{\zeta_S^0 - n_2 - n_1' + 1} }
    {1 - \omega^{\zeta_E^0 + n_2' - n_1' + 1 } }
    \\
    \label{eq:R-recurrence-iii}
    \rmat{n_1,}{n_2}{n_1'}{n_2'}
    &=
    \rmat{n_1,}{n_2 -1}{n_1'}{n_2'}
    \omega^{\alpha_2 + \mu_2 + 1}
    \frac
    {1 - \omega^{\zeta_W^0 - 1 + n_2 - n_1} }
    {1 -  \omega^{\zeta_S^0 + n_2 - n_1'}}
    \\ 
    \label{eq:R-recurrence-iv}
    \rmat{n_1}{n_2}{n_1'}{n_2'}
    &=
    \rmat{n_1 - 1,}{n_2}{n_1'}{n_2'}
    \omega^{\alpha_{1} - \mu_1 - 1}
    \frac
    {1 - \omega^{\zeta_N^0 + n_2' - n_1 +1 }}
    {1 -  \omega^{\zeta_W^0 - 1 + n_2 - n_1 +1} } 
  \end{align}
  which determine $R$ uniquely up to an overall scalar.
\end{theorem}
\begin{proof}
  We found the first recurrence above.
  The other three relations can be derived in a similar way by applying \ref{eq:intertwiner-rel} to $y_2$ and \ref{eq:intertwiner-rel-inv} to  $y_1^{-1}$ and $y_2$.
  The relations are clearly sufficient to determine $R$ up to a scalar, and then existence and uniqueness follow from Theorem \ref{thm:R-matrix-exists}.
\end{proof}

\subsection{The normalized \texorpdfstring{\(R\)}{R}-matrix}

We need to define some more log-parameters associated to a log-colored crossing \(C\).
Set \(\epsilon = 1\) if \(C\) is positive and \(\epsilon = -1\) if it is negative.
If \(C\) is not pinched, 
\begin{equation*}
  K= 
  \frac{e^{2\pi i \gamma_{\lN}}}{1 - \left(b_{2'}/b_{1}\right)^{\epsilon}}
\end{equation*}
is not \(\infty\), so we can choose \(\kappa \in \mathbb{C}\) with
\[
  e^{2\pi i \kappa} = K.
\]
We will show in \cref{thm:R-matrix transformation} that the \(R\)-matrix is independent of the choice of \(\kappa\), so we do not consider it to be part of a log-coloring.
Extending our earlier abbreviations, define
  \begin{align}
    \label{eq:flattening-N}
    \zeta_{\lN}^0
    &=
    \epsilon (\beta_{2'} - \beta_1)
    &
    \zeta_{\lN}^1
    &=
    ( \kappa - \gamma_N)
    \\
    \label{eq:flattening-W}
    \zeta_{\lW}^0
    &=
    \epsilon(\beta_{2} - \beta_{1} - \mu_1 )
    &
    \zeta_{\lW}^1
    &=
    ( \kappa - \gamma_{\lW} + \epsilon \mu_1)
    \\
    \label{eq:flattening-S}
    \zeta_{\lS}^0
    &=
    \epsilon(\beta_{2} - \beta_{1'} + \mu_2 - \mu_1 )
    &
    \zeta_{\lS}^1
    &=
    (\kappa - \gamma_{\lS} + \epsilon (\mu_1 - \mu_2))
    \\
    \label{eq:flattening-E}
    \zeta_{\lE}^0
    &=
    \epsilon(\beta_{2'} - \beta_{1'} + \mu_2 )
    &
    \zeta_{\lE}^1
    &=
    ( \kappa - \gamma_{\lE} - \epsilon \mu_2)
  \end{align}
For \(j \in \set{{\lN}, {\lW}, {\lS}, {\lE}}\) we have
\[
  e^{2\pi i \zeta_j^1} = 
  \frac{
    1
    }{
    1 - e^{2\pi i \zeta_j^0}
  }
\]
so for \(n \in \ZZ\) we can consider the \defemph{quantum dilogarithm}
\(
  \qlf{\zeta^0, \zeta^1}{n}
\)
discussed in \cref{sec:qlog-definitions}.
For \(n \ge 0\),
\begin{equation}
  \label{eq:qlf recurrence}
  \qlf{\zeta^0, \zeta^1}{n}
  =
  \qlf{\zeta^0, \zeta^1}{0}
  \frac{
    \omega^{-n\zeta^1}
  }{
    (1 - \omega^{\zeta^0 + 1})
    \cdots
    (1 - \omega^{\zeta^0 + n})
  }
\end{equation}
is an \(n\)-independent scalar times a power of \(\omega^{\zeta^1}\) times a shifted \(q\)-factorial at \(q = \omega\).
The relation \(e^{2\pi i \zeta^1} = 1/(1 - e^{2\pi i \zeta^0})\) ensures \(\qlfname\) is periodic modulo \(\nr\) in the integer argument.

\begin{definition}
  \label{def:R-mat}
  The \defemph{\(R\)-matrix} associated to a positive, log-colored crossing is
  \begin{equation}
    \label{eq:R-mat-positive}
    \begin{aligned}
      \widehat R_{n_1 n_2}^{n_1' n_2'}
      =
      {}
      &\frac{
        \omega^{-(\nr-1)(\zeta_{\lW}^0 + \zeta_{\lW}^1)}
      }{
        \nr
      }
      \\
      &\times
      \omega^{n_2 - n_1}
      \frac{
        \qlf{\zeta_{\lN}^0, \zeta_{\lN}^1}{n_2' - n_1}
        \qlf{\zeta_{\lS}^0, \zeta_{\lS}^1}{n_2 - n_1'}
        }{
          \qlf{\zeta_{\lW}^0, \zeta_{\lW}^1}{n_2 - n_1 - 1}
          \qlf{\zeta_{\lE}^0, \zeta_{\lE}^1}{n_2' - n_1'}
      }.
    \end{aligned}
  \end{equation}
  and to a negative crossing is
  \begin{equation}
    \label{eq:R-mat-negative}
    \begin{aligned}
      \rmatm{n_1}{n_2}{n_1'}{n_2'}
      &=
      \frac{
        \omega^{(\nr-1)(\zeta_{\lE}^0 + \zeta_{\lE}^1 - \zeta_{\lS}^0 -  \zeta_{\lS}^1 -  \zeta_{\lN}^0 - \zeta_{\lN}^1)}
        }{
        \nr
      }
      \\
      &\phantom{=}\times
      \omega^{n_1 - n_2}
      \frac{
        \qlf{\zeta_{\lW}^0, \zeta_{\lW}^1}{n_1 - n_2}
        \qlf{\zeta_{\lE}^0, \zeta_{\lE}^1}{n_1' - n_2' - 1}
      }{
        \qlf{\zeta_{\lN}^0, \zeta_{\lN}^1}{n_1 - n_2' - 1}
        \qlf{\zeta_{\lS}^0, \zeta_{\lS}^1}{n_1' - n_2 - 1}
      }
    \end{aligned}
  \end{equation}

  We call the \(\U_\xi\)-module morphism
  \begin{equation}
    \label{eq:braiding}
    \tau R : 
    \rep{\hat{\chi}_{1}}
    \otimes
    \rep{\hat{\chi}_{2}}
    \to
    \rep{\hat{\chi}_{2'}}
    \otimes
    \rep{\hat{\chi}_{1'}},
    \vbh{n_{1} n_{2}}
    \mapsto
    \sum_{n_{1}' n_{2}'}
    \rmat{n_{1}}{n_{2}}{n_{1}'}{n_{2}'}
    \vbh{n_{2}' n_{1}'}
  \end{equation}
  the \defemph{braiding} associated to a positive log-colored crossing.
  For a negative crossing the braiding is \(\overline{R} \tau\).
\end{definition}

The relation \eqref{eq:qlf recurrence} (and some algebra) immediately show that the positive \(R\)-matrix satisfies the relations of \cref{eq:intertwiner-rel}.
One can check directly that the negative \(R\)-matrix is its inverse for appropriate parameter values \cite{McPhailSnyderVolume}.
Our assumption that the crossing is not pinched ensures we avoid the poles and zeros of \(\qlfname\), which occur at \(\zeta^0 \in \ZZ\).
In \cref{sec:pinched-limit} we show that despite these poles the \(R\)-matrix is regular in the pinched limit.

\subsection{Factorization of the \texorpdfstring{\(R\)}{R}-matrix}
\label{sec:factorization}

A remarkable property of the \(R\)-matrices in \cref{def:R-mat} is that they factor into four operators, each defined in terms of a quantum dilogarithm.

\begin{figure}
  \centering
  \subcaptionbox{The initial triangulation.\label{fig:triangulated-braiding-0}}{ \def\svgwidth{2.5in} 
\begingroup%
  \makeatletter%
  \providecommand\color[2][]{%
    \errmessage{(Inkscape) Color is used for the text in Inkscape, but the package 'color.sty' is not loaded}%
    \renewcommand\color[2][]{}%
  }%
  \providecommand\transparent[1]{%
    \errmessage{(Inkscape) Transparency is used (non-zero) for the text in Inkscape, but the package 'transparent.sty' is not loaded}%
    \renewcommand\transparent[1]{}%
  }%
  \providecommand\rotatebox[2]{#2}%
  \newcommand*\fsize{\dimexpr\f@size pt\relax}%
  \newcommand*\lineheight[1]{\fontsize{\fsize}{#1\fsize}\selectfont}%
  \ifx\svgwidth\undefined%
    \setlength{\unitlength}{215.81824493bp}%
    \ifx\svgscale\undefined%
      \relax%
    \else%
      \setlength{\unitlength}{\unitlength * \real{\svgscale}}%
    \fi%
  \else%
    \setlength{\unitlength}{\svgwidth}%
  \fi%
  \global\let\svgwidth\undefined%
  \global\let\svgscale\undefined%
  \makeatother%
  \begin{picture}(1,1.03093509)%
    \lineheight{1}%
    \setlength\tabcolsep{0pt}%
    \put(0.366521,0.92532356){\color[rgb]{0,0,0}\makebox(0,0)[lt]{\lineheight{1.25}\smash{\begin{tabular}[t]{l}$P_+$\end{tabular}}}}%
    \put(0.38424276,0.03278304){\color[rgb]{0,0,0}\makebox(0,0)[lt]{\lineheight{1.25}\smash{\begin{tabular}[t]{l}$P_-$\end{tabular}}}}%
    \put(0.1486229,0.51022397){\color[rgb]{0,0,0}\makebox(0,0)[lt]{\lineheight{1.25}\smash{\begin{tabular}[t]{l}$P_2$\end{tabular}}}}%
    \put(0.59268972,0.51326773){\color[rgb]{0,0,0}\makebox(0,0)[lt]{\lineheight{1.25}\smash{\begin{tabular}[t]{l}$P_1$\end{tabular}}}}%
    \put(0,0){\includegraphics[width=\unitlength,page=1]{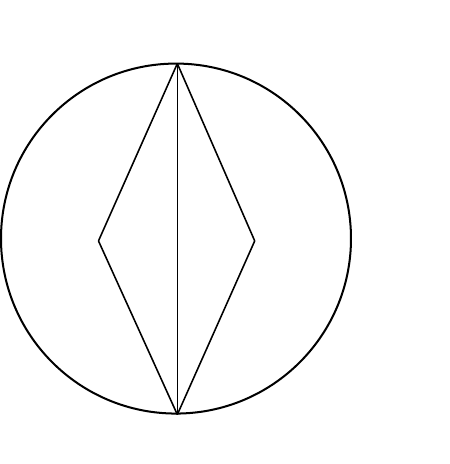}}%
  \end{picture}%
\endgroup%
 }%
  \hfill
  \subcaptionbox{Building a tetrahedron on top of a quadrilateral.\label{fig:triangulated-braiding-1}}{ \def\svgwidth{2.5in} 
\begingroup%
  \makeatletter%
  \providecommand\color[2][]{%
    \errmessage{(Inkscape) Color is used for the text in Inkscape, but the package 'color.sty' is not loaded}%
    \renewcommand\color[2][]{}%
  }%
  \providecommand\transparent[1]{%
    \errmessage{(Inkscape) Transparency is used (non-zero) for the text in Inkscape, but the package 'transparent.sty' is not loaded}%
    \renewcommand\transparent[1]{}%
  }%
  \providecommand\rotatebox[2]{#2}%
  \newcommand*\fsize{\dimexpr\f@size pt\relax}%
  \newcommand*\lineheight[1]{\fontsize{\fsize}{#1\fsize}\selectfont}%
  \ifx\svgwidth\undefined%
    \setlength{\unitlength}{215.81824493bp}%
    \ifx\svgscale\undefined%
      \relax%
    \else%
      \setlength{\unitlength}{\unitlength * \real{\svgscale}}%
    \fi%
  \else%
    \setlength{\unitlength}{\svgwidth}%
  \fi%
  \global\let\svgwidth\undefined%
  \global\let\svgscale\undefined%
  \makeatother%
  \begin{picture}(1,1.03093509)%
    \lineheight{1}%
    \setlength\tabcolsep{0pt}%
    \put(0.366521,0.92532356){\color[rgb]{0,0,0}\makebox(0,0)[lt]{\lineheight{1.25}\smash{\begin{tabular}[t]{l}$P_+$\end{tabular}}}}%
    \put(0.38424276,0.03278304){\color[rgb]{0,0,0}\makebox(0,0)[lt]{\lineheight{1.25}\smash{\begin{tabular}[t]{l}$P_-$\end{tabular}}}}%
    \put(0.1486229,0.51022397){\color[rgb]{0,0,0}\makebox(0,0)[lt]{\lineheight{1.25}\smash{\begin{tabular}[t]{l}$P_2$\end{tabular}}}}%
    \put(0.59268972,0.51326773){\color[rgb]{0,0,0}\makebox(0,0)[lt]{\lineheight{1.25}\smash{\begin{tabular}[t]{l}$P_1$\end{tabular}}}}%
    \put(0,0){\includegraphics[width=\unitlength,page=1]{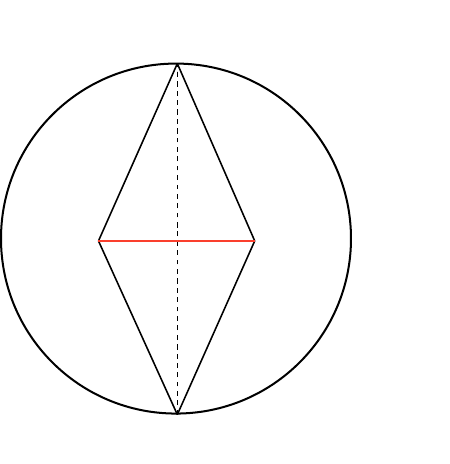}}%
  \end{picture}%
\endgroup%
 }%
  \hfill
  \subcaptionbox{Adding two more tetrahedra.\label{fig:triangulated-braiding-2}}{ \def\svgwidth{2.5in} 
\begingroup%
  \makeatletter%
  \providecommand\color[2][]{%
    \errmessage{(Inkscape) Color is used for the text in Inkscape, but the package 'color.sty' is not loaded}%
    \renewcommand\color[2][]{}%
  }%
  \providecommand\transparent[1]{%
    \errmessage{(Inkscape) Transparency is used (non-zero) for the text in Inkscape, but the package 'transparent.sty' is not loaded}%
    \renewcommand\transparent[1]{}%
  }%
  \providecommand\rotatebox[2]{#2}%
  \newcommand*\fsize{\dimexpr\f@size pt\relax}%
  \newcommand*\lineheight[1]{\fontsize{\fsize}{#1\fsize}\selectfont}%
  \ifx\svgwidth\undefined%
    \setlength{\unitlength}{215.81824493bp}%
    \ifx\svgscale\undefined%
      \relax%
    \else%
      \setlength{\unitlength}{\unitlength * \real{\svgscale}}%
    \fi%
  \else%
    \setlength{\unitlength}{\svgwidth}%
  \fi%
  \global\let\svgwidth\undefined%
  \global\let\svgscale\undefined%
  \makeatother%
  \begin{picture}(1,1.03093509)%
    \lineheight{1}%
    \setlength\tabcolsep{0pt}%
    \put(0.366521,0.92532356){\color[rgb]{0,0,0}\makebox(0,0)[lt]{\lineheight{1.25}\smash{\begin{tabular}[t]{l}$P_+$\end{tabular}}}}%
    \put(0.38424276,0.03278304){\color[rgb]{0,0,0}\makebox(0,0)[lt]{\lineheight{1.25}\smash{\begin{tabular}[t]{l}$P_-$\end{tabular}}}}%
    \put(0.1486229,0.51022397){\color[rgb]{0,0,0}\makebox(0,0)[lt]{\lineheight{1.25}\smash{\begin{tabular}[t]{l}$P_2$\end{tabular}}}}%
    \put(0.59268972,0.51326773){\color[rgb]{0,0,0}\makebox(0,0)[lt]{\lineheight{1.25}\smash{\begin{tabular}[t]{l}$P_1$\end{tabular}}}}%
    \put(0,0){\includegraphics[width=\unitlength,page=1]{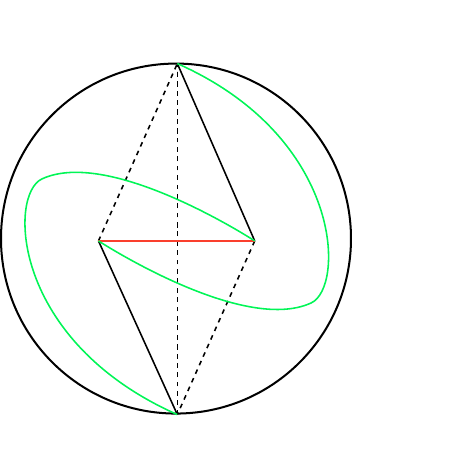}}%
  \end{picture}%
\endgroup%
 }%
  \hfill
  \subcaptionbox{The final result.\label{fig:triangulated-braiding-total}}{ \def\svgwidth{2.5in} 
\begingroup%
  \makeatletter%
  \providecommand\color[2][]{%
    \errmessage{(Inkscape) Color is used for the text in Inkscape, but the package 'color.sty' is not loaded}%
    \renewcommand\color[2][]{}%
  }%
  \providecommand\transparent[1]{%
    \errmessage{(Inkscape) Transparency is used (non-zero) for the text in Inkscape, but the package 'transparent.sty' is not loaded}%
    \renewcommand\transparent[1]{}%
  }%
  \providecommand\rotatebox[2]{#2}%
  \newcommand*\fsize{\dimexpr\f@size pt\relax}%
  \newcommand*\lineheight[1]{\fontsize{\fsize}{#1\fsize}\selectfont}%
  \ifx\svgwidth\undefined%
    \setlength{\unitlength}{215.81824493bp}%
    \ifx\svgscale\undefined%
      \relax%
    \else%
      \setlength{\unitlength}{\unitlength * \real{\svgscale}}%
    \fi%
  \else%
    \setlength{\unitlength}{\svgwidth}%
  \fi%
  \global\let\svgwidth\undefined%
  \global\let\svgscale\undefined%
  \makeatother%
  \begin{picture}(1,1.03093509)%
    \lineheight{1}%
    \setlength\tabcolsep{0pt}%
    \put(0.366521,0.92532356){\color[rgb]{0,0,0}\makebox(0,0)[lt]{\lineheight{1.25}\smash{\begin{tabular}[t]{l}$P_+$\end{tabular}}}}%
    \put(0.38424276,0.03278304){\color[rgb]{0,0,0}\makebox(0,0)[lt]{\lineheight{1.25}\smash{\begin{tabular}[t]{l}$P_-$\end{tabular}}}}%
    \put(0.1486229,0.51022397){\color[rgb]{0,0,0}\makebox(0,0)[lt]{\lineheight{1.25}\smash{\begin{tabular}[t]{l}$P_2$\end{tabular}}}}%
    \put(0.59268972,0.51326773){\color[rgb]{0,0,0}\makebox(0,0)[lt]{\lineheight{1.25}\smash{\begin{tabular}[t]{l}$P_1$\end{tabular}}}}%
    \put(0,0){\includegraphics[width=\unitlength,page=1]{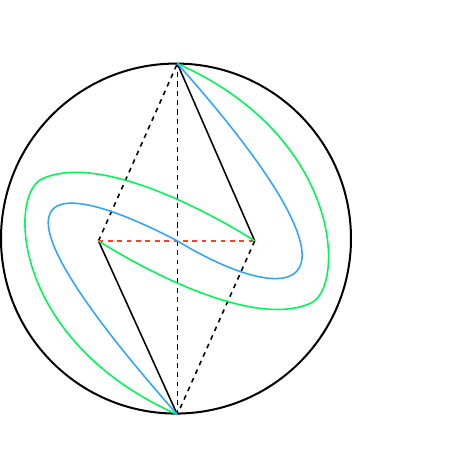}}%
  \end{picture}%
\endgroup%
 }%
  \caption{Building an ideal octahedron.}
  \label{fig:build-ideal-octahedron}
\end{figure}

\begin{theorem}
  \label{thm:R-factorization}
  At a positive, non-pinched, log-colored crossing the brading matrix factors as
  \begin{equation}
    \label{eq:R-factorization-positive}
    \frac{
      1
    }{
      N
    }
    \mathcal{Z}_{\lE} (\mathcal{Z}_{\lN} \otimes \mathcal{Z}_{\lS}) \mathcal{Z}_{\lW}
  \end{equation}
  where
  \begin{align*}
    \mathcal{Z}_{\lE}(\widehat{v}_{n_1 n_2})
    &=
    \frac{
      1
    }{
      \Lambda({\zeta^0_{\lE} , \zeta^1_{\lE}}|{n_1 - n_2})
    }
    \widehat{v}_{n_1 n_2}
    \\
    \mathcal{Z}_{\lW}(\widehat{v}_{n_1 n_2})
    &=
      \omega^{-(N-1)(\zeta_{\lW}^0 + \zeta_{\lW}^1)}
    \frac{
      \omega^{n_2 - n_1} 
    }{
      \Lambda({\zeta^0_{\lW}, \zeta^1_{\lW}}|{n_2 - n_1 -1})
    }
    \widehat{v}_{n_1 n_2}
    \\
    \mathcal{Z}_{\lN}(\widehat{v}_{n})
    &=
    \sum_{n'=0}^{N -1}
    \Lambda({\zeta^0_{\lN}, \zeta^1_{\lN}}|{n'-n})\widehat{v}_{n'}
    \\
    \mathcal{Z}_{\lS}(\widehat{v}_{n})
    &=
    \sum_{n'=0}^{N -1}
    \Lambda({\zeta^0_{\lS}, \zeta^1_{\lS}}|{n-n'})\widehat{v}_{n'}
  \end{align*}

  At a negative crossing we instead have a factorization
  \begin{equation}
    \label{eq:R-factorization-negative}
    \frac{
      1
    }{
      N
    }
    \overline{\mathcal{Z}}_{\lE} (\overline{\mathcal{Z}}_{\lS} \otimes \overline{\mathcal{Z}}_{\lN}) \overline{\mathcal{Z}}_{\lW}
  \end{equation}
  in terms of operators
  \begin{align*}
    \overline{\mathcal{Z}}_{\lE}(\widehat{v}_{n_1 n_2})
    &=
    \omega^{-(N -1)(\zeta_{\lE}^0 + \zeta_{\lE}^1)}
    \qlf{\zeta_{\lE}^0, \zeta_{\lE}^1}{n_1 - n_2 - 1}
    \widehat{v}_{n_1 n_2}
    \\
    \overline{\mathcal{Z}}_{\lW}(\widehat{v}_{n_1 n_2})
    &=
    \omega^{n_2 - n_1}
    \qlf{\zeta_{\lW}^0, \zeta_{\lW}^1}{n_2 - n_1} \widehat{v}_{n_1 n_2}
    \\
    \overline{\mathcal{Z}}_{\lN}(\widehat{v}_n)
    &=
    \omega^{(N -1)(\zeta_{\lN}^0 + \zeta_{\lN}^1)}
    \sum_{n'}
    \frac{
      1
    }{
      \qlf{\zeta_{\lN}^0, \zeta_{\lN}^1}{n - n' - 1}
    }
    \widehat{v}_{n'}
    \\
    \overline{\mathcal{Z}}_{\lS}(\widehat{v}_n)
    &=
    \omega^{(N -1)(\zeta_{\lS}^0 + \zeta_{\lS}^1)}
    \sum_{n'}
    \frac{
      1
    }{
      \qlf{\zeta_{\lS}^0, \zeta_{\lS}^1}{n' - n - 1}
    }
    \widehat{v}_{n'}
  \end{align*}
\end{theorem} 

  In both cases the overall matrix does not depend on the choice of \(\kappa\) for the crossing but the operator factors \(\mathcal{Z}_i\) do.

\begin{proof}
  \Cref{eq:R-factorization-positive} follows immediately from \cref{eq:R-mat-positive}, and similarly for the negative case.
  (Recall that the braiding is given by the \(R\)-matrix followed by a flip map.)
\end{proof}

\begin{marginfigure}
  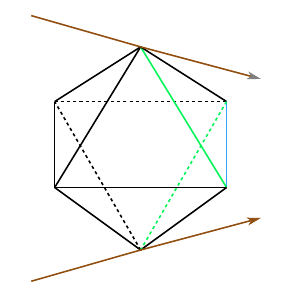
  \caption{
    A side view of the ideal octahedron in \cref{fig:triangulated-braiding-total}.
    The grey dashed edges indicate the identification of the ideal vertices \(P_{+}, P_{+}'\) and \(P_{-}, P_{-}'\).
  }
  \label{fig:octahedron-side}
\end{marginfigure}

This factorization has a natural geometric interpretation.
Thinking of a crossing as coming out of the page we can associate it to the triangulated punctured disc shown in \cref{fig:triangulated-braiding-0}.
The braiding can then be implemented as a series of flips in this triangulation, and each operator factor is associated to a flip.
From a \(3\)-dimensional point of view each flip can be thought of as building an ideal tetrahedron, so the braiding corresponds to a (twisted) ideal octahedron, as in \cref{fig:build-ideal-octahedron}.
We show a side view of the octahedron in \cref{fig:octahedron-side}.
This leads directly to the \defemph{octahedral decomposition} \cite{Kim2018}, an ideal triangulation of the tangle complement associated to a choice of diagram.

The parameters of the \(\chi\)-coloring are geometrically natural: they parametrize hyperbolic structures on the octahedral decomposition \cite{McPhailSnyder2022,McPhailSnyder2024} and can be used to can be used to directly compute the Chern-Simons invariant via a sum of dilogarithms \cite{Cho2014,McPhailSnyderVolume}.
Our \(R\)-matrix is analogous: instead of a scalar-valued dilogarithm for each factor we have a matrix-valued quantum dilogarithm.
This analogy is discussed further in \cite{McPhailSnyderVolume}.

The crossing is pinched exactly when the ideal tetrahedra are geometrically degenerate.
In this case the log-parameters \(\zeta_j^0\) lie in \(\ZZ\), so the quantum dilogarithms have zeros and poles.
In \cref{thm:R-mat-pinched} we show that the overall \(R\)-matrix has a well-defined limit, namely Kashaev's \(R\)-matrix \cite{Kashaev1995}.
However the factorization only holds in a weaker sense: we still have a product of four operators \(\mathcal{Z}_j\), each of which depends on only two of the tensor indices, but there is now a cutoff term \(\theta_{n_1 n_2}^{n_1' n_2'} \in \set{0, 1}\) involving every index at the crossing.

\subsection{The holonomy Yang-Baxter equation}
\label{sec:braid groupoids}

Next we show our \(R\)-matrices satisfy braid relations.
We discuss how to use these braid relations to define tangle and link invariants in \cite{McPhailSnyderVolume}, extending \cite{Kashaev2005}.
The holonomy \(R\)-matrix satisfies a \emph{family} of braid relations parametrized by all admissible choices of colors, which have a natural interpretation in terms of flat \(\sla\) connections on the braid complement, equivalently representations of the fundamental group of the complement into \(\slg\).
Our matrices depend on an additional local choice of log-coloring whose geometric meaning is less clear.
However, it turns out that they depend only on certain quantities derived from the log-coloring which have an interpretation in terms of the holonomy of the longitude.
These \defemph{log-decorations} are discussed in detail in \cite{McPhailSnyderVolume}; here we say just enough to prove the braid relation.

\begin{definition}
  \label{def:RIII-admissible}
  Choosing characters \(\hat \chi_{1}, \hat \chi_{2}, \hat \chi_{3} \in \Wcharhat\) for the top segments determines a \(\chi\)-coloring of the \(\RIII\) diagrams in \cref{fig:RIII-before-and-after-chis}.
  We say the \(\chi\)-coloring is \defemph{\(\RIII\)-admissible} if both colorings are admissible.%
  \note{%
    The colorings of the bottom strands match because \(\R\) and thus \(B\) satisfy braid relations.%
  }
  We say that log-colorings for each diagram are \defemph{\(\RIII\)-compatible} if the log-parameters for the boundary regions and segments agree and \(\beta + \beta'' = \beta' + \tilde{\beta}'\) (in the notation of \cref{fig:RIII-before-and-after-middle}).
\end{definition}

\begin{figure}
  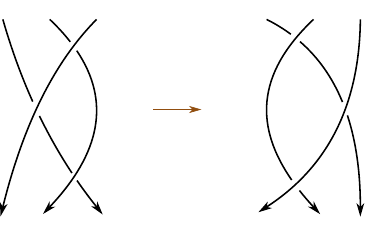
  \caption{
    \(\chi\)-colorings of the braid diagrams occurring in an \(\RIII\) move.
  }
  \label{fig:RIII-before-and-after-chis}
\end{figure}

\begin{figure}
\begingroup%
  \makeatletter%
  \providecommand\color[2][]{%
    \errmessage{(Inkscape) Color is used for the text in Inkscape, but the package 'color.sty' is not loaded}%
    \renewcommand\color[2][]{}%
  }%
  \providecommand\transparent[1]{%
    \errmessage{(Inkscape) Transparency is used (non-zero) for the text in Inkscape, but the package 'transparent.sty' is not loaded}%
    \renewcommand\transparent[1]{}%
  }%
  \providecommand\rotatebox[2]{#2}%
  \newcommand*\fsize{\dimexpr\f@size pt\relax}%
  \newcommand*\lineheight[1]{\fontsize{\fsize}{#1\fsize}\selectfont}%
  \ifx\svgwidth\undefined%
    \setlength{\unitlength}{177.67268372bp}%
    \ifx\svgscale\undefined%
      \relax%
    \else%
      \setlength{\unitlength}{\unitlength * \real{\svgscale}}%
    \fi%
  \else%
    \setlength{\unitlength}{\svgwidth}%
  \fi%
  \global\let\svgwidth\undefined%
  \global\let\svgscale\undefined%
  \makeatother%
  \begin{picture}(1,0.65227594)%
    \lineheight{1}%
    \setlength\tabcolsep{0pt}%
    \put(0,0){\includegraphics[width=\unitlength,page=1]{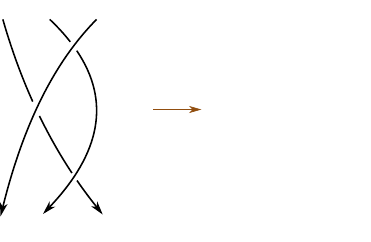}}%
    \put(0.09998844,0.62074315){\makebox(0,0)[lt]{\lineheight{1.25}\smash{\begin{tabular}[t]{l}$\beta$\end{tabular}}}}%
    \put(0.28271139,0.33166973){\makebox(0,0)[lt]{\lineheight{1.25}\smash{\begin{tabular}[t]{l}$\beta'$\end{tabular}}}}%
    \put(0.074551,0.01174866){\makebox(0,0)[lt]{\lineheight{1.25}\smash{\begin{tabular}[t]{l}$\beta''$\end{tabular}}}}%
    \put(0,0){\includegraphics[width=\unitlength,page=2]{RIII-before-and-after-middle.pdf}}%
    \put(0.83074933,0.62083637){\makebox(0,0)[lt]{\lineheight{1.25}\smash{\begin{tabular}[t]{l}$\beta$\end{tabular}}}}%
    \put(0.74831887,0.33105816){\makebox(0,0)[lt]{\lineheight{1.25}\smash{\begin{tabular}[t]{l}$\tilde \beta'$\end{tabular}}}}%
    \put(0.81229327,0.00441912){\makebox(0,0)[lt]{\lineheight{1.25}\smash{\begin{tabular}[t]{l}$\beta''$\end{tabular}}}}%
  \end{picture}%
\endgroup%

  \caption{
    A log-colored \(\RIII\) move is allowed when the log-parameters of the boundary regions and segments are the same on both sides and additionally \(\beta + \beta'' = \beta' + \tilde{\beta}'\).
    This ensures that the log-longitudes of each component are preserved.
  }
  \label{fig:RIII-before-and-after-middle}
\end{figure}

This condition has a geometric meaning.
Recall that the parameter \(m\) of a color \(\chi = (a, b, m)\) represents a distinguished eigenvalue of the associated meridian holonomy \(\psi(\chi) \in \slg\).
For a \(\chi\)-colored link diagram this choice gives an associated longitude eigenvalue \(\ell\) \cite{McPhailSnyder2022} for each component; the choice of \((m, \ell)\) versus \((m^{-1}, \ell^{-1})\) is called a \defemph{decoration}.
A log-coloring determines a logarithm \(\mu\) of \(m = e^{2\pi i \mu}\) and also a logarithm \(\lambda\) of \(\ell  = e^{2 \pi i \lambda}\) by taking a certain signed sum of segment log-parameters \(\beta_{i}\).
These definitions can be extended to tangles like those in \cref{fig:RIII-before-and-after-middle}, and we call the choice of \((\mu, \lambda)\) for each component of the tangle a \defemph{log-decoration}.
General definitions and more details are given in \cite{McPhailSnyderVolume}.
Here it suffices to consider the condition in \cref{def:RIII-admissible}, which asserts that the log-decorations on both sides of the \(\RIII\) move are the same.

\begin{theorem}
  \label{thm:RIII}
  Fix any \(\RIII\)-admissible \(\chi\)-colorings and compatible log-colorings of the diagrams \(D, D'\) in \cref{fig:RIII-before-and-after-chis}.
  By assigning the holonomy braidings of \cref{def:R-mat} to the crossings of the diagrams in \cref{fig:RIII-before-and-after-chis} we obtain \(\U_{\xi}\)-module intertwiners
  \[
    \J_{\xi}(D)
    ,
    \J_{\xi}(D')
    :
    \rep{\hat{\chi}_{1}}
    \otimes
    \rep{\hat{\chi}_{2}}
    \otimes
    \rep{\hat{\chi}_{3}}
    \to
    \rep{\hat{\chi}_{3''}}
    \otimes
    \rep{\hat{\chi}_{2''}}
    \otimes
    \rep{\hat{\chi}_{1''}}
  \]
  Then the holonomy \(R\)-matrices satisfy braid relations in the sense that \(\J_{\xi}(D) = \J_{\xi}(D')\).
\end{theorem}

\begin{figure}
  \centering
\begingroup%
  \makeatletter%
  \providecommand\color[2][]{%
    \errmessage{(Inkscape) Color is used for the text in Inkscape, but the package 'color.sty' is not loaded}%
    \renewcommand\color[2][]{}%
  }%
  \providecommand\transparent[1]{%
    \errmessage{(Inkscape) Transparency is used (non-zero) for the text in Inkscape, but the package 'transparent.sty' is not loaded}%
    \renewcommand\transparent[1]{}%
  }%
  \providecommand\rotatebox[2]{#2}%
  \newcommand*\fsize{\dimexpr\f@size pt\relax}%
  \newcommand*\lineheight[1]{\fontsize{\fsize}{#1\fsize}\selectfont}%
  \ifx\svgwidth\undefined%
    \setlength{\unitlength}{109.25290205bp}%
    \ifx\svgscale\undefined%
      \relax%
    \else%
      \setlength{\unitlength}{\unitlength * \real{\svgscale}}%
    \fi%
  \else%
    \setlength{\unitlength}{\svgwidth}%
  \fi%
  \global\let\svgwidth\undefined%
  \global\let\svgscale\undefined%
  \makeatother%
  \begin{picture}(1,2.93188782)%
    \lineheight{1}%
    \setlength\tabcolsep{0pt}%
    \put(0,0){\includegraphics[width=\unitlength,page=1]{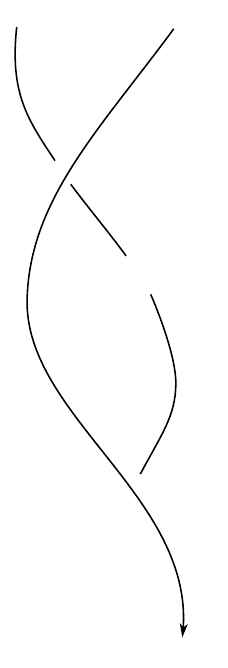}}%
    \put(0.69802025,2.87625433){\makebox(0,0)[lt]{\lineheight{1.25}\smash{\begin{tabular}[t]{l}$\chi_1$\end{tabular}}}}%
    \put(0.32282748,2.87624433){\makebox(0,0)[lt]{\lineheight{1.25}\smash{\begin{tabular}[t]{l}$\chi_2$\end{tabular}}}}%
    \put(-0.00606632,2.87320708){\makebox(0,0)[lt]{\lineheight{1.25}\smash{\begin{tabular}[t]{l}$\chi_3$\end{tabular}}}}%
    \put(0,0){\includegraphics[width=\unitlength,page=2]{RIII-colored-together.pdf}}%
    \put(0.69800818,0.02031652){\makebox(0,0)[lt]{\lineheight{1.25}\smash{\begin{tabular}[t]{l}$\chi_1$\end{tabular}}}}%
    \put(0.32281541,0.02030623){\makebox(0,0)[lt]{\lineheight{1.25}\smash{\begin{tabular}[t]{l}$\chi_2$\end{tabular}}}}%
    \put(-0.0060782,0.01726928){\makebox(0,0)[lt]{\lineheight{1.25}\smash{\begin{tabular}[t]{l}$\chi_3$\end{tabular}}}}%
  \end{picture}%
\endgroup%

  \caption{
    Invariance under \(\RIII\) moves follows from showing that \(\J_{\xi}\) is the identity map on this diagram for every choice of coloring and any log-coloring with trivial log-decoration and matching values on the boundary.
    A coloring of the whole diagram is determined by the choice of colors \(\chi_1, \chi_2, \chi_3\) on the top.
  }
  \label{fig:RIII-colored-together}
\end{figure}

The rest of this section is devoted to the proof of this theorem.
It is convenient to instead prove that
\[
  \J_{\xi}(D_{\RIII}) : 
  \rep{\hat{\chi}_{1}}
  \otimes
  \rep{\hat{\chi}_{2}}
  \otimes
  \rep{\hat{\chi}_{3}}
  \to
  \rep{\hat{\chi}_{1}}
  \otimes
  \rep{\hat{\chi}_{2}}
  \otimes
  \rep{\hat{\chi}_{3}}
\]
is the identity map, where \(D_{\RIII}\) is the diagram of \cref{fig:RIII-colored-together} with a \(\chi\)-coloring and log-coloring derived from \(D\) and \(D'\) in the obvious way.
The compatibility condition on the log-parameters of the boundary regions and segments is given in \cref{fig:RIII-log-parameters}.
\cref{thm:RIII} is equivalent to the claim that \(\J_{\xi}(D_{\RIII})\) is the identity map for any choice of log-parameters for the internal regions and segments as long as the log-decoration condition
\begin{equation}
  \label{eq:RIII beta condition}
  \beta_{2} - \beta_{2'} + \beta_{2''} - \beta_{2'''} = 0
\end{equation}
is satisfied.

\begin{figure}
  \centering
  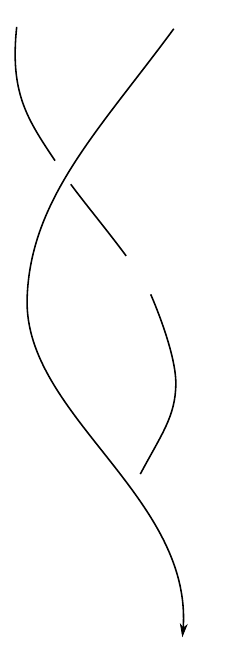
  \caption{
    We require the log-parameters of the boundary regions and segments to be compatible as shown.
    We allow any choice of internal log-parameters as long as \(\beta_{2} - \beta_{2'} + \beta_{2''} - \beta_{2'''} = 0\).
  }
  \label{fig:RIII-log-parameters}
\end{figure}

To prove \cref{thm:RIII} we need to show the space of allowed log-colorings of \(D_{\RIII}\) is connected.
Consider first the space \(A_{3}^{0}\) of \(\RIII\)-admissible triples of colors \((\chi_{1}, \chi_{2}, \chi_{3})\) and set
\[
  A_3 = A_3^0 \times ( \mathbb{C} \setminus \set{0} ).
\]
(the last factor is for keeping track of \(e^{2 \pi i \gamma}\)).
Let \(\widetilde{A}_3\) be the set of points 
\[
  (\gamma, \alpha_1, \alpha_2, \alpha_3, \beta_1, \beta_2, \beta_3, \mu_1, \mu_2, \mu_3)
  \in \CC^{10}
\]
such that if we set
\begin{align*}
  \chi_j =
  \left(
    e^{2\pi i \alpha_{j}}, 
    e^{2\pi i \beta_j},
    e^{2\pi i \mu_j}
  \right)
  \text{ for \(j = 1, 2, 3\) and }
  r = e^{2\pi i \gamma}
\end{align*}
then \((\chi_1, \chi_2, \chi_3, r) \in A_3\).

\begin{lemma}
  \(\widetilde{A}_{3}\) is path-connected.
\end{lemma}

\begin{proof}
  \(A_3\) is obtained by removing finitely many algebraic hypersurfaces from \(\CC^{10}\), so it is path-connected.
  \(\widetilde{A}_{3}\) is a covering space of \(A_{3}\) built by repeatedly taking covers of the same type as the connected infinite cyclic cover
  \[
    \zeta \mapsto \exp(2\pi i \zeta) , \CC \to \CC \setminus \set{0}
  \]
  occurring as the Riemann surface of the complex logarithm, so \(\widetilde{A}_3\) is path-connected as well.
\end{proof}

\begin{lemma}
  \label{thm:DR3-well-defined}
  \(\J_{\xi}(D_{\RIII})\) is a well-defined, matrix-valued function on \(\widetilde{A}_{3}\).
\end{lemma}

\begin{proof}
  Suppose \(\mathfrak{f}\) is a log-coloring of an admissible \(\chi\)-coloring \(\chi\) of \(D_{\RIII}\).
  Let \(\mathfrak{f}'\) be another log-coloring with the same log-parameters on the boundary as \(\mathfrak{f}\).
  Set 
  \[
    \lambda
    =
    \beta_{2} - \beta_{2'} + \beta_{2''} - \beta_{2'''} = 0
  \]
  where the \(\beta_{i}\) are the segment parameters of \(\mathfrak{f}\), and similarly define \(\lambda'\) for \(\mathfrak{f}'\).
  We show that
  \begin{equation}
    \label{eq:jfunc-longitude-dep}
    \jfunc(D_{\RIII}, \chi, \mathfrak{f})
    =
    \omega^{- (\lambda' - \lambda)\mu_{2}}
    \jfunc(D_{\RIII}, \chi, \mathfrak{f})
  \end{equation}
  This is a special case of a more general transformation rule given in \cite{McPhailSnyderVolume}.
  In particular if \(\mathfrak{f}\) and \(\mathfrak{f}'\) are \(\RIII\)-compatible log-colorings corresponding to the same point of \(\widetilde{A}_{3}\) then \(\lambda = \lambda' = 0\), so  \(
    \jfunc(D_{\RIII}, \chi, \mathfrak{f})
    =
    \jfunc(D_{\RIII}, \chi, \mathfrak{f})
  \) as claimed.

  The linear map \(\jfunc\) has an explicit description as a state-sum: each segment of \(D_{\RIII}\) is associated the \(\U_\xi\)-module \(\rep{\alpha, \beta, \mu}\) with log-parameters determined by \(\mathfrak{f}\).
  These log-parameters determine a basis \(\set{\vbh{n} | n \in \mathbb{Z}/\nr \mathbb{Z}}\) of each module and \(R\)-matrices associated to each crossing, and the action of \(\jfunc(D, \chi, \mathfrak{f})\) is given by summing over all internal indices.

  Consider a region of \(D_{\RIII}\) with log-parameter \(\gamma\).
  By \cref{thm:R-matrix transformation}, changing the value of \(\gamma\) to \(\gamma + k\) changes each of the \(R\)-matrices adjacent to the region.
  There are two changes to understand: one depending on the segment index and one global.
  For the first factor, observe that in the state-sum defining \(\jfunc(D, \chi, \mathfrak{f})\) the sum for a segment \(i\) adjacent to the region with parameter \(\gamma\) picks up a factor of the form
  \[
    \omega^{ \pm n_i (k - k)} = 1
  \]
  where \(n_i\) is the index assigned to the segment.
  Here the two factors of \(k\) are from the crossings on each side of the segment.
  As the signs of the \(n_i k_j\) terms in \cref{eq:R-matrix transformation:gamma} depend only on whether the segment is incoming or outgoing these always cancel as claimed.

  The second global factor is more geometric.
  Suppose our region is adjacent to crossings \(1, \dots, p\).
  Each corner corresponds to an ideal tetrahedron of orientation \(\epsilon_i\), and the term \(\Gamma\) in \cref{eq:R-matrix transformation:gamma} says that the state-sum has a global change by
  \[
    \omega^{\frac{k}{2}( \epsilon_{1} \zeta_i^0 + \epsilon_2 \zeta_2^0 + \cdots  + \epsilon_p \zeta_p^0)}.
  \]
  Our claim follows from showing that
  \begin{equation}
    \label{eq:edge-flattening}
    \epsilon_{1} \zeta_i^0 + \epsilon_2 \zeta_2^0 + \cdots  + \epsilon_p \zeta_p^0 = 0.
  \end{equation}
  This is the logarithmic gluing equation of the edge of the octahedral decomposition lying in this region: saying that it holds is part of the claim that the parameters of \cref{eq:flattening-N,eq:flattening-W,eq:flattening-S,eq:flattening-E} define a flattening.
  One can confirm \eqref{eq:edge-flattening} by following the same argument as in the proof of \cite[Theorem 3.2]{McPhailSnyder2022}; see also \cite[Figure 13]{McPhailSnyder2022}.%
  \note{
    Taking arbitrary logarithms \(\beta_i, \gamma_j\) of the \(\chi\)-coordinates automatically gives a flattening; if we instead took arbitrary logarithms of the shape parameters \(z_i^j = e^{2 \pi i \zeta_i^j}\) this would not be the case.
    This property is characteristic of Ptolemey-type coordinates of hyperbolic ideal triangulations \cite{Zickert2007,Zickert2016}.
  }

  Now consider the effect of changing a segment log-parameter \(\beta\) with associated index \(n\).
  Changing from \(\beta\) to \(\beta + l\) picks up index shifts in both adjacent \(R\)-matrices, which cancel, and also picks up a factor involving \(\gamma\)s and \(\mu\)s.
  The  \(\gamma\) terms always cancel segment-by-segment.
  The \(\mu\) terms for the log-parameters of components \(1\) and \(3\) always cancel as well, while for component \(2\) an analysis of the relative signs gives \cref{eq:jfunc-longitude-dep}.
\end{proof}

\begin{lemma}
  \label{thm:scalar-lemma}
  There is a continuous function \(\Upsilon : \widetilde A_3 \to \CC^{\times}\) such that
  \begin{equation}
    \label{eq:Upsilon-def}
    \jfunc(D_{\RIII}, \widetilde a) = \Upsilon(\widetilde a) \id
  \end{equation}
  for each \(\widetilde a \in \widetilde A_3\).
\end{lemma}
\begin{proof}
  Recall from \cref{thm:R-matrix-exists} that the \(R\)-matrices  defining \(\jfunc\) are characterized up to a scalar as intertwiners of the outer \(R\)-matrix \(\R\) (an algebra automorphism).
  Because \(\R\) satisfies a colored Yang-Baxter equation the \(R\)-matrices do as well, at least up to a scalar.
  This scalar is \(\Upsilon\), and because \(\jfunc\) is continuous it must be as well.
\end{proof}

\begin{lemma}
  \label{thm:cocyle-condition}
  For any \(\widetilde{a} \in \widetilde{A}_{3}\)
  \[
    \det \jfunc(D_{\RIII}, \widetilde a) = \pm 1
    .\qedhere
  \]
\end{lemma}
\begin{proof}
  The determinant of \(\jfunc(D_{\RIII}, \widetilde a)\) is the product of the determinants of braidings for each of the \(6\) crossings raised to the \(\nr\)th power.
  It thus suffices to show that the product of these determinants is \(\pm 1\).
  In \cref{thm:R-mat-det} we showed that the determinant of the braiding at a log-colored crossing \(c\) of sign \(\epsilon\) is
  \begin{equation*}
    \det \tau R
    =
    C^{\epsilon \nr^2}
    \left[
      e^{\pi i (\gamma_{\lW} - \gamma_{\lE}) }
      e^{2\pi i(\lambda_{1} + \lambda_{2})}
      e^{4\pi i \epsilon(\mu_{1} + \mu_{2})}
    \right]^{\nr(\nr-1)}
    e^{- \nr I(c) / 2\pi i}
    .
  \end{equation*}
  There are \(5\) factors in the determinant:
  \begin{enumerate}
    \item
      \(C = \nr/\df{0}^2 = \nr \exp\left( -2 \nr^{-1} \sum_{k=0}^{\nr-1} k \log(1 - \omega^{k}) \right)\) is a constant depending only on \(\nr\),
    \item
      \(\gamma_{\lE}\), \(\gamma_{\lW}\) are two of the region log-parameters at the crossing,
    \item \(\mu_1\) and \(\mu_{2}\) are the log-meridians of the two components at the crossing,
    \item
      \(\lambda_{1}\) and \(\lambda_{2}\) are the contributions of the crossing to the log-longitude, and
    \item 
      \(I(c)\) is a sum of four dilogarithms.
  \end{enumerate}

  We can check by elementary methods that the product of terms 1, 2, and 3 over \(D_{\RIII}\) is \(1\).
  The product of term \(4\) over the diagram is \(1\) because of our assumption that the total log-longitudes in our log-coloring are \(0\).
  Dealing with the remaining term \(5\) is more complicated: the fact that it gives \(\pm 1\) is a consequence of the fact that the complex Chern-Simons invariant is well-defined and given by Neumann's sum of lifted dilogarithms \cite{Neumann2004}; algebraically, this is because the lifted dilogarithm obeys a \(5\)-term relation.
  A definition of the complex Chern-Simons invariant for tangles and a detailed proof of invariance is given in \cite{McPhailSnyderVolume}, and this implies that the product of term \(5\) over the diagram is \(\pm 1\) as claimed.
\end{proof}

\begin{proof}[Proof of \cref{thm:RIII}]
  We need to show that \(\Upsilon = 1\) uniformly.
  Taking the determinant of both sides of \cref{eq:Upsilon-def} and applying \cref{thm:cocyle-condition} shows
  \[
    \det \jfunc(D_{\RIII}, \widetilde{a})
    =
    \Upsilon(\widetilde{a})^{\nr^3}
    =
    \pm 1
  \]
  for every \(\widetilde{a} \in \widetilde{A}_3\).
  We conclude that \(\Upsilon\) takes values in the discrete set \[\set{w \in \CC\given w^{\nr^3} = -1}\] of \(\nr^3\)th roots of \(-1\).
  Thus \(\Upsilon\) is a discrete-valued, continuous function on a connected space, so it must be constant.
  To show that this constant is \(1\) we simply need to check that \(\Upsilon(\widetilde{a}_0) = 1\) for some \(\widetilde{a}_0 \in \widetilde{A}_3\).

  We choose a point \(\widetilde{a}_{0}\) defining the Kashaev invariant.
  The underlying \(\chi\)-coloring of \(\widetilde{a}_0\) has \(\chi_1 = \chi_3 = (-1,1,-1)\) and \(\chi_2 = (-1,-1,-1)\).
  To get a point of \(\widetilde{A}_3\) we need to choose (part of) a log-coloring.
  With the labeling of \cref{fig:RIII-log-parameters} we set \(\mu_1, \mu_2, \mu_3 = -1/2\), \(\gamma_0 = 0\), \(\alpha_1 = \alpha_2 = \alpha_3 = -1/2\), \(\beta_1 = 0\), \(\beta_2 = -1/2\), and \(\beta_3 = -1\).

  We want to compute \(\jfunc(D_{\RIII}, \widetilde{a}_0)\).
  To do so we need to choose a log-coloring of the internal segments and regions of the diagram.
  Because all the crossings are pinched is is simplest to choose a \emph{standard} log-coloring as in \cref{eq:standard-flattening}.
  This is possible because of our choices of \(\beta_1, \beta_2, \beta_3\) above.
  Choosing a standard log-coloring also ensures we satisfy \cref{eq:RIII beta condition}.
  By \cref{thm:R-mat-pinched} (alternatively by \cref{thm:pinched-R-matrix-dual}) the \(R\)-matrices in our diagram are all the same as the \(R\)-matrix defining the Kashaev invariant.
  This matrix is well-known to satisfy the braid relation exactly, so we conclude that \(\jfunc(D_{\RIII}, \widetilde{a}_0) = \id\) and \(\Upsilon(\widetilde{a}_0) = 1\).
\end{proof}

\section{\texorpdfstring{\(R\)}{R}-matrix computations}
\label{sec:R-matrix-computations}
This section contains computations with the \(R\)-matrices.
We refer frequently to the identities concerning shifted \(q\)-factorials and quantum dilogarithms collected in \cref{sec:quantum-dilogarithms}.

\begin{lemma}
  \label{thm:R-mat-phi-form}
  The \(R\)-matrix coefficients \eqref{eq:R-mat-positive} associated to a positive, log-colored crossing can be written as
  \begin{equation}
    \label{eq:R-mat-positive-phi-form}
    \begin{aligned}
      \rmat{n_1}{n_2}{n_1'}{n_2'}
    &=
    \frac{
      \omega^{\Theta/2}
      (\omega^{-\nr \zeta_{\lW}^{0}} -1)
      }{
      \nr 
    }
    \frac{
      \pf{\zeta_{\lN}^0 + n_2' - n_1}
      \pf{\zeta_{\lS}^0 + n_2 - n_1'}
      }{
      \pf{\zeta_{\lE}^0 + n_2' - n_1'}
      \pf{\zeta_{\lW}^0 + n_2 - n_1 - 1}
    }
    \\
    &\phantom{=}\times
    \omega^{\zeta_{\lW}^{0} + n_{2} - n_{1}}
    \omega^{
      n_1(\alpha_{1}  -\mu_1)
      +
      n_2(\alpha_{2} + \mu_2)
      -
      n_1'(\alpha_{1'} -\mu_1)
      -
      n_2'(\alpha_{2'} + \mu_2)
    }
  \end{aligned}
  \end{equation}
  Similarly for a negative crossing the matrix coefficients \eqref{eq:R-mat-negative} are given by 
  \begin{equation}
    \label{eq:R-mat-negative-phi-form}
    \begin{aligned}
      \rmatm{n_1}{n_2}{n_1'}{n_2'}
    &=
    \frac{
      \omega^{-\Theta/2}
      (\omega^{-\nr \zeta_{\lW}^{0}} -1)
      }{
      \nr 
    }
    \frac{
      \pf{-\zeta_{\lN}^0 + n_2' - n_1}
      \pf{-\zeta_{\lS}^0 + n_2 - n_1'}
      }{
      \pf{-\zeta_{\lE}^0 + n_2' - n_1'}
      \pf{-\zeta_{\lW}^0 + n_2 - n_1 - 1}
    }
    \\
    &\phantom{=}\times
    \omega^{\zeta_{\lW}^{0} + n_{1} - n_{2}}
    \omega^{
      n_1(\alpha_{1}  + \mu_1)
      +
      n_2(\alpha_{2} - \mu_2)
      -
      n_1'(\alpha_{1'} + \mu_1)
      -
      n_2'(\alpha_{2'} - \mu_2)
    }
    \\
    &\phantom{=}\times
    \omega^{
      -
      (\zeta_{\lW}^{0} - \zeta_{\lS}^{0} + n_{1} - n_{1}')
      (\zeta_{\lW}^{0} - \zeta_{\lN}^{0} + n_{2} - n_{2}')
    }
    \end{aligned}
  \end{equation}
  where in both cases
  \begin{equation*}
    \Theta
    \defeq
    \zeta_{\lW}^{0} \zeta_{\lW}^{1}
    +
    \zeta_{\lE}^{0} \zeta_{\lE}^{1}
    -
    \zeta_{\lN}^{0} \zeta_{\lN}^{1}
    -
    \zeta_{\lS}^{0} \zeta_{\lS}^{1}
    .
    \qedhere
  \end{equation*}
\end{lemma}
\begin{proof}
  The positive case \eqref{eq:R-mat-positive-phi-form} follows form the definition \eqref{eq:qlf-def} of \(\qlf{\zeta^{0}, \zeta^{1}}{n}\) and some elementary algebra.
  For the negative case \eqref{eq:R-mat-negative-phi-form} we must also use the relation \(\zeta_{\lN}^{0} + \zeta_{\lS}^{0} = \zeta_{\lW}^{0} + \zeta_{\lE}^{0}\) and the inversion relation \cite[equation (16)]{Faddeev2001}, which in our conventions is
  \begin{equation}
    \label{eq:pf inversion relation}
    \pf{\zeta} \pf{- \zeta - 1} 
    =
    (1 - \omega^{- \nr \zeta})
    \exp\left(
      \frac{
        \pi i (\nr + 1/\nr)
        }{
        12
      }
      - \frac{\pi i}{\nr} \left[ \zeta - \frac{\nr -1}{2} \right]^{2}
    \right)
    .
  \end{equation}

\end{proof}

\subsection{Log-parameter dependence}
\label{sec:log-paramater-dependence}

Here we compute the dependence of the \(R\)-matrix on the choice of log-coloring as used in the proof of \cref{thm:RIII}.

\begin{lemma}
  \label{thm:R-matrix transformation}
  \begin{thmenum}
  \item
    \label{thm:R-matrix transformation:kappa}
    The \(R\)-matrix coefficients of \cref{def:R-mat} are independent of the choice of \(\kappa\).
  \item
    \label{thm:R-matrix transformation:gamma}
    Let \(\mathfrak{f}\) and \(\tilde{\mathfrak{f}}\) be log-colorings of a crossing for which the \(\mu_i\) and \(\beta_i\) are the same and the \(\gamma_i\) differ by integers:
    \begin{align*}
      \tilde \mu_j &= \mu_j
                   &
      \tilde \beta_j &= \beta_j
                     &
      \tilde \gamma_j &= \gamma_j + k_j
    \end{align*}
    Then the coefficients of the \(R\)-matrix are related by
    \begin{equation}
      \label{eq:R-matrix transformation:gamma}
      \rmat[\tilde{\mathfrak{f}}]{n_{1}}{n_{2}}{n_{1}'}{n_{2}'}
      =
      \omega^{
        \frac{1}{2}
        \Gamma(\tilde{\mathfrak{f}}, \mathfrak{f}) 
      }
      \omega^{
        k_{\lN}(n_{2}' - n_{1})
            +
            k_{\lS}(n_{2} - n_{1}')
            -
            k_{\lW}(n_{2} - n_{1})
            -
            k_{\lE}(n_{2}' - n_{1}')
          }
          \rmat[\mathfrak{f}]{n_{1}}{n_{2}}{n_{1}'}{n_{2}'}
      \end{equation}
      where
      \[
        \Gamma(\tilde{\mathfrak{f}}, \mathfrak{f})
          =
          k_{\lN}\zeta_{\lN}^{0}
          +
          k_{\lS}\zeta_{\lS}^{0}
          -
          k_{\lW}\zeta_{\lW}^{0}
          -
          k_{\lE}\zeta_{\lE}^{0}
      \]
    \item
      \label{thm:R-matrix transformation:beta}
      Let \(\mathfrak{f}\) and \(\tilde{\mathfrak{f}}\) be log-colorings of a crossing for which the \(\mu_i\) and \(\gamma_i\) are the same and the \(\beta_i\) differ by integers:
      \begin{align*}
        \tilde \mu_j &= \mu_j
                     &
        \tilde \beta_j &= \beta_j + l_j
                       &
        \tilde \gamma_j &= \gamma_j
      \end{align*}
      Then the \(R\)-matrix coefficients are related by
      \begin{equation}
        \label{eq:R-matrix transformation:beta}
          \rmat[\tilde{\mathfrak{f}}]{n_{1}}{n_{2}}{n_{1}'}{n_{2}'}
          =
          \omega^{
            \frac{1}{2} \mathrm{B}(\tilde{\mathfrak{f}}, \mathfrak{f})
          }
          \rmat[\mathfrak{f}]{n_{1} + l_{1},}{n_{2} + l_{2}}{n_{1}' + l_{1}', }{n_{2}' + l_{2}'}
      \end{equation}
      where
      \[
        \begin{aligned}
          \mathrm{B}(\tilde{\mathfrak{f}}, \mathfrak{f})
          &=
          l_{2'}(\gamma_{\lE} - \gamma_{\lW} - \epsilon_{2'} \mu_{2})
          +
          l_{1'}(\gamma_{\lS} - \gamma_{\lE} - \epsilon_{1'} \mu_{2})
          \\
          &\phantom{=}
          +
          l_{2}(\gamma_{\lE} - \gamma_{\lW} - \epsilon_{2} \mu_{2})
          +
          l_{1}(\gamma_{\lS} - \gamma_{\lE} - \epsilon_{1} \mu_{2})
        \end{aligned}
      \]
      and the \(\epsilon_i\) are the signs of \cref{fig:crossing-log-dec-signs}.
  \end{thmenum}
\end{lemma}

\begin{figure}
  \centering
\begingroup%
  \makeatletter%
  \providecommand\color[2][]{%
    \errmessage{(Inkscape) Color is used for the text in Inkscape, but the package 'color.sty' is not loaded}%
    \renewcommand\color[2][]{}%
  }%
  \providecommand\transparent[1]{%
    \errmessage{(Inkscape) Transparency is used (non-zero) for the text in Inkscape, but the package 'transparent.sty' is not loaded}%
    \renewcommand\transparent[1]{}%
  }%
  \providecommand\rotatebox[2]{#2}%
  \newcommand*\fsize{\dimexpr\f@size pt\relax}%
  \newcommand*\lineheight[1]{\fontsize{\fsize}{#1\fsize}\selectfont}%
  \ifx\svgwidth\undefined%
    \setlength{\unitlength}{277.34116745bp}%
    \ifx\svgscale\undefined%
      \relax%
    \else%
      \setlength{\unitlength}{\unitlength * \real{\svgscale}}%
    \fi%
  \else%
    \setlength{\unitlength}{\svgwidth}%
  \fi%
  \global\let\svgwidth\undefined%
  \global\let\svgscale\undefined%
  \makeatother%
  \begin{picture}(1,0.47412055)%
    \lineheight{1}%
    \setlength\tabcolsep{0pt}%
    \put(0,0){\includegraphics[width=\unitlength,page=1]{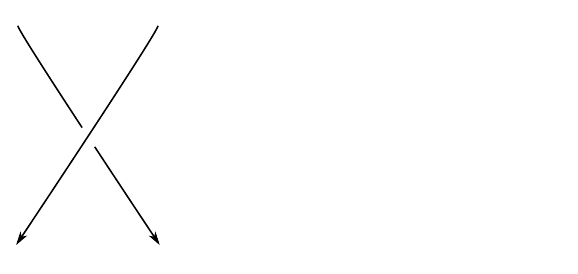}}%
    \put(0.26092847,0.44174785){\makebox(0,0)[lt]{\lineheight{1.25}\smash{\begin{tabular}[t]{l}$-$\end{tabular}}}}%
    \put(0.00722541,0.44174352){\makebox(0,0)[lt]{\lineheight{1.25}\smash{\begin{tabular}[t]{l}$+$\end{tabular}}}}%
    \put(0.26092847,0.00586272){\makebox(0,0)[lt]{\lineheight{1.25}\smash{\begin{tabular}[t]{l}$-$\end{tabular}}}}%
    \put(0.00722541,0.00635921){\makebox(0,0)[lt]{\lineheight{1.25}\smash{\begin{tabular}[t]{l}$+$\end{tabular}}}}%
    \put(0.81543474,0.44174785){\makebox(0,0)[lt]{\lineheight{1.25}\smash{\begin{tabular}[t]{l}$+$\end{tabular}}}}%
    \put(0.56173168,0.44174352){\makebox(0,0)[lt]{\lineheight{1.25}\smash{\begin{tabular}[t]{l}$-$\end{tabular}}}}%
    \put(0.81543474,0.00586272){\makebox(0,0)[lt]{\lineheight{1.25}\smash{\begin{tabular}[t]{l}$+$\end{tabular}}}}%
    \put(0.56173168,0.00635921){\makebox(0,0)[lt]{\lineheight{1.25}\smash{\begin{tabular}[t]{l}$-$\end{tabular}}}}%
    \put(0,0){\includegraphics[width=\unitlength,page=2]{crossing-log-dec-signs.pdf}}%
  \end{picture}%
\endgroup%

  \caption{Signs for the dependence of the \(R\)-matrix on the segment log-parameters \(\beta_{i}\).}
  \label{fig:crossing-log-dec-signs}
\end{figure}

\begin{proof}
  \ref{thm:R-matrix transformation:kappa}
  This is immediate from \cref{eq:qlf-shift-1}.
  Consider a positive crossing (the negative case is similar).
  If we choose \(\kappa + p\) instead of \(\kappa\) for \(p \in \ZZ\), by \eqref{eq:qlf-shift-1} the new matrix coefficients are
  \begin{align*}
      &
      \frac{
        \omega^{-(\nr-1)(\zeta_{\lW}^0 + \zeta_{\lW}^1 + p)}
      }{
        \nr
      }
      \omega^{n_2 - n_1}
      \frac{
        \qlf{\zeta_{\lN}^0, \zeta_{\lN}^1 + p}{n_2' - n_1}
        \qlf{\zeta_{\lS}^0, \zeta_{\lS}^1 + p}{n_2 - n_1'}
        }{
          \qlf{\zeta_{\lW}^0, \zeta_{\lW}^1 + p}{n_2 - n_1 - 1}
          \qlf{\zeta_{\lE}^0, \zeta_{\lE}^1 + p}{n_2' - n_1'}
      }
      \\
      &=
      \omega^{p - p(\zeta_{\lN}^0 + \zeta_{\lS}^0 - \zeta_{\lW}^0 - \zeta_{\lE}^0) - p(n_2' - n_1 + n_2 - n_1' - (n_2 - n_1 - 1) -(n_2' - n_1'))}
      \rmat{n_1}{n_2}{n_1'}{n_2'}
      \\
      &=
      \rmat{n_1}{n_2}{n_1'}{n_2'}
  \end{align*}
  so we get the same coefficients as for our original choice of \(\kappa\).

  \ref{thm:R-matrix transformation:gamma}
  Suppose the crossing is positive.
  We can use \cref{eq:qlf-shift-1} to find
  \begin{align*}
    \rmat[\tilde{\mathfrak{f}}]{n_1}{n_2}{n_1'}{n_2'}
    &=
    \frac{
      \omega^{-(\nr-1)(\zeta_{\lW}^0 + \zeta_{\lW}^1 - k_{\lW})}
    }{
      \nr
    }
    \omega^{n_2 - n_1}
    \frac{
      \qlf{\zeta_{\lN}^0, \zeta_{\lN}^1 - k_{\lN}}{n_2' - n_1}
      \qlf{\zeta_{\lS}^0, \zeta_{\lS}^1 - k_{\lS}}{n_2 - n_1'}
    }{
      \qlf{\zeta_{\lW}^0, \zeta_{\lW}^1 - k_{\lW}}{n_2 - n_1 - 1}
      \qlf{\zeta_{\lE}^0, \zeta_{\lE}^1 - k_{\lE}}{n_2' - n_1'}
    }
    \\
    &=
    \omega^{- k_{\lW}}
    \omega^{
      \frac{1}{2}
      (
      k_{\lN}\zeta_{\lN}^{0}
      +
      k_{\lS}\zeta_{\lS}^{0}
      -
      k_{\lW}\zeta_{\lW}^{0}
      -
      k_{\lE}\zeta_{\lE}^{0}
      )
    }
    \omega^{
      k_{\lN}(n_{2}' - n_{1})
      +
      k_{\lS}(n_{2} - n_{1}')
      -
      k_{\lW}(n_{2} - n_{1} - 1)
      -
      k_{\lE}(n_{2}' - n_{1}')
    }
    \rmat[\mathfrak{f}]{n_1}{n_2}{n_1'}{n_2'}
  \end{align*}
  as claimed, and a similar computation works for negative crossings.

  \ref{thm:R-matrix transformation:beta}
  Again suppose the crossing is positive.
  \Cref{eq:qlf-shift-0} shows that
  \[
    \qlf{\zeta_{\lN}^0 + l_{2}' - l_{1}, \zeta_{\lN}^1}{n_2' - n_1 - l_{2}' + l_{1}}
    =
    \omega^{
      \frac{1}{2}
      (l_{2}' - l_{1}) \zeta_{\lN}^{1}
    }
    \qlf{\zeta_{\lN}^0, \zeta_{\lN}^1}{n_2' - n_1}
  \]
  and repeating for the other factors shows that
  \begin{align*}
    \rmat[\tilde{\mathfrak{f}}]{n_1 - l_{1},}{n_2 - l_{2}}{n_1' - l_{1}', }{n_2' - l_{2}'}
    &=
    \omega^{
      \frac{1}{2}
      \left[
      l_{2}'( \zeta_{\lN}^{1} - \zeta_{\lE}^{1} )
      +
      l_{1}'( \zeta_{\lE}^{1} - \zeta_{\lS}^{1} )
      +
      l_{2}( \zeta_{\lS}^{1} - \zeta_{\lW}^{1} )
      +
      l_{1}( \zeta_{\lW}^{1} - \zeta_{\lN}^{1} )
    \right]
    }
    \rmat[\mathfrak{f}]{n_1}{n_2}{n_1'}{n_2'}
    .
  \end{align*}
  (Here the power of \(\omega\) from \(\omega^{n_{2} - n_{1}}\) cancels with that from \(\omega^{-(\nr-1)\zeta_{\lW}^{0}}\).)
  For a positive crossing \cref{eq:flattening-N,eq:flattening-E} give
  \[
    \zeta_{\lN}^{1} - \zeta_{\lE}^{1}
    =
    (\kappa - \gamma_{\lN}) - ( \kappa - \gamma_{\lE} - \mu_{2})
    =
    \gamma_{\lE} - \gamma_{\lN} + \mu_{2}
  \]
  as claimed, and for a positive crossing \cref{fig:crossing-log-dec-signs} says that \(\epsilon_{2} = +1\).
  Similar computations give the coefficients of the other \(l_i\) and the negative crossing case.
\end{proof}

\subsection{The pinched limit of the \texorpdfstring{\(R\)}{R}-matrix}
\label{sec:pinched-limit}

In the computation of the \(R\)-matrix in \cref{sec:R-matrix} we excluded certain singular (``pinched'') configurations of characters.
We now explain how to derive the \(R\)-matrix there as a limit of the generic case.

\begin{definition}
  \label{def:standard-flattening}
  Consider a \(\chi\)-colored crossing as in \cref{fig:crossing-regions}.
  If any of the relations
  \begin{equation}
    \label{eq:pinched-crossing-conditions}
    b_2 = m_1 b_1
    ,\quad
    m_2 b_2 = m_1 b_{1'}
    ,\quad
    b_{2'} = b_{1}
    ,\quad
    m_2 b_{2'} = b_{1'}
  \end{equation}
  hold it is not hard to show that all of them do.
  In this case we say the crossing is \defemph{pinched}.
  There is an obvious way to choose logarithms \(\beta_i\) of the parameters \(b_i\) at a pinched crossing; we say a log-coloring of such a crossing is \defemph{standard} if there is a \(\beta\) so that
  \begin{align}
    \label{eq:standard-flattening}
    \beta_{1} &= \beta,
              &
    \beta_{2} &= \beta + \mu_1,
              &
    \beta_{1'} &= \beta + \mu_2\text{, and}
               &
    \beta_{2'} &= \beta.
  \end{align}
  This is equivalent to requiring
  \(\zeta_{\lN}^0 = \zeta_{\lW}^0 = \zeta_{\lS}^0 = \zeta_{\lE}^0 = 0\)
  in terms of the parameters (\ref{eq:flattening-N}--\ref{eq:flattening-E}).
\end{definition}

\begin{remark}
  \label{rem:pinched eigenlines}
  This definition has a geometric interpretation discussed in more detail in \cite{McPhailSnyder2024}.
  Consider a braid diagram \(D\) and write \(\pi\) for the fundamental group of the complement of the braid.
  Via the Wirtinger presentation \(\pi\) is generated by meridians.
  A \(\chi\)-coloring of \(D\) determines both a representation \(\rho : \pi \to \slg\) and a \defemph{decoration} of \(\rho\), which is a distinguished eigenspace \(L \in \mathbb{C}P^1\) for the image of each meridian under \(\rho\) (generically \(g \in \slg\) has two eigenspaces, and we pick one).
  A crossing is pinched exactly when the eigenspaces \(L_{1}, L_{2}\) of the incoming overstrand and understrand are equal.
\end{remark}

\begin{remark}
  \label{rem:standard-flattening}
  We can use \cref{thm:R-matrix transformation} to write the \(R\)-matrix of a nonstandard log-coloring in terms of a standard one with some index shifts and scalar factors that are nonsingular in the pinched limit.
  As such, we consider only standard pinched log-colorings.
\end{remark}

For \(k \in \ZZ\) we write \(\modb{k}\) for the unique integer with
\[
  0 \le \modb{k} < \nr
  \text{ and }
  \modb{k} \equiv k \pmod \nr
\]
We also write
\[
  \cutoff{k}
  =
  \begin{cases}
    1 & 0 \le k \le \nr -1
    \\
    0 & \text{ otherwise,}
  \end{cases}
  =
  \begin{cases}
    1 & k = \modb{k}
    \\
    0 & \text{ otherwise.}
  \end{cases}
\]
Recall the \(q\)-Pochhammer symbol \(\qp{a}{n}\) of \eqref{eq:qp-def}, which satisfies
\[
  \qp{\omega}{n} = (1 - \omega) \cdots (1 - \omega^{n -1}) \text{ for } n \ge 0.
\]

\begin{theorem}
  \label{thm:R-mat-pinched}
  \begin{thmenum}
    \item
    \label{thm:R-mat-pinched:regular}
     The \(R\)-matrices \eqref{eq:R-mat-positive} are well-defined in the pinched limit where \(e^{2\pi i \zeta_j^0} \to 1\).
    \item
    \label{thm:R-mat-pinched:formula}
      At a positive, pinched crossing with a standard log-coloring the matrix coefficients are given by
      \begin{equation}
        \label{eq:R-mat-pinched:formula}
        \begin{aligned}
          \rmat{n_1}{n_2}{n_1'}{n_2'} 
          &=
          \frac{1}{\nr}
          \theta_{n_1 n_2}^{n_1' n_2'}
          A_{n_1 n_2}^{n_1' n_2'}
          \\
          &\phantom{=}\times
          \omega^{
            n_1(\alpha_1-\mu_1 - 1)
            +
            n_2(\alpha_2 + \mu_2 + 1)
            -
            n_1' (\alpha_{1'} - \mu_{1})
            -
            n_2' (\alpha_{2'} + \mu_2)
          }
          \\
          &\phantom{=}\times
          \frac{
            \qp{\omega}{\modb{n_2' - n_1'}}
            \qp{\omega}{\modb{n_2 - n_1 - 1}}
            }{
            \qp{\omega}{\modb{n_2' - n_1}}
            \qp{\omega}{\modb{n_2 - n_1'}}
          }.
        \end{aligned}
      \end{equation}
    Here
    \[
      \theta_{n_1 n_2}^{n_1' n_2'}
      =
      \cutoff{
        \modb{n_1 - n_2} + \modb{n_1' - n_2' -1}
      }
      \cutoff{
        \modb{n_2' - n_1} + \modb{n_2 - n_1'}
      }
    \]
    and
    \[
      A_{n_1 n_2}^{n_1' n_2'}
      =
      \frac{a_{1'}}{a_1}
      \left( \frac{a_1}{ m_1} \right)^{2 - \cutoff{n_1 - n_2} - \cutoff{n_2 - n_1'}}
      (a_2 m_2)^{- \cutoff{n_2 - n_1'}}
      (a_{2'} m_2)^{1- \cutoff{n_1' - n_2' - 1}}
      .
    \]
  \item
    \label{thm:R-mat-pinched:Kashaev}
    In particular, when \(\alpha_i = \mu_i = -1/2 \text{ or } (\nr - 1)/2\) for all \(i\) we obtain
    \begin{equation}
      \label{eq:Kashaev-R-matrix}
      \rmat{n_1}{n_2}{n_1'}{n_2'} 
      =
      \theta_{n_1 n_2}^{n_1' n_2'}
      \frac{
        \nr
        \omega^{
          n_2' - n_1
          +1/2
        }
        }{
        \qp{\omega}{\modb{n_2' - n_1}}
        \qp{\omega}{\modb{n_2 - n_1'}}
        \qp{\overline{\omega}}[\overline{\omega}]{\modb{n_1' - n_2' - 1}}
        \qp{\overline{\omega}}[\overline{\omega}]{\modb{n_1 - n_2}}
      }
    \end{equation}
    Up to complex conjugation this is Kashaev's \(R\)-matrix \cite{Kashaev1995}.
    Specifically, \eqref{eq:Kashaev-R-matrix} is \((R_K)_{n_2' n_1'}^{n_1 n_2}\) at \(q = \overline{\omega}\), where \(R_K\) is the matrix given in \cite[Section 4.1]{Murakami2018} .
    \qedhere
  \end{thmenum}
\end{theorem}

As discussed in \cref{sec:factorization} the pinched \(R\)-matrices almost factor into four pieces: the obstruction is the term \(\theta_{n_1 n_2}^{n_1' n_2'}\).

\begin{proof}
  We prove regularity for positive crossings; the negative case works similarly.
  By \cref{rem:standard-flattening} we may assume the log-coloring is standard.
  By \cref{thm:R-mat-phi-form} the matrix coefficients are
  \[
    \begin{aligned}
      \rmat{n_1}{n_2}{n_1'}{n_2'}
    &=
    \frac{
      \omega^{\Theta/2 + (1 - \nr)\zeta_{\lW}^{0}}
      }{
      \nr 
    }
    \frac{
      \pf{\zeta_{\lN}^0}
      \pf{\zeta_{\lS}^0}
      }{
      \pf{\zeta_{\lE}^0}
      \pf{\zeta_{\lW}^0}
    }
    \\
    &\phantom{=}\times
    \omega^{
      n_1(\alpha_{1}  -\mu_1 - 1)
      +
      n_2(\alpha_{2} + \mu_2 + 1)
      -
      n_1'(\alpha_{1'} -\mu_1)
      -
      n_2'(\alpha_{2'} + \mu_2)
    }
    \\
    &\phantom{=}\times
    ( 1 - \omega^{\nr \zeta_{\lW}^{0}} )
    \frac{
      \qlog{\zeta_{\lN}^0}{n_2' - n_1}
      \qlog{\zeta_{\lS}^0}{n_2 - n_1'}
      }{
      \qlog{\zeta_{\lW}^0}{n_2 - n_1 - 1}
      \qlog{\zeta_{\lE}^0}{n_2' - n_1'}
    }
  \end{aligned}
  \]
    Because our log-coloring is standard we are taking the limit where each \(\zeta_j^0 \to 0\), holding the \(\alpha_{i}\) and \(\mu_{k}\) fixed.
    We immediately have
    \[
    \frac{
      \omega^{\Theta/2 + (1 - \nr) \zeta_{\lW}^{0}}
      }{
      \nr 
    }
    \frac{
      \pf{\zeta_{\lN}^0}
      \pf{\zeta_{\lS}^0}
      }{
      \pf{\zeta_{\lE}^0}
      \pf{\zeta_{\lW}^0}
    }
      \to
      \frac{
        1
        }{
        \nr
      }
    \]
    so it remains to understand
    \begin{equation}
      \label{eq:pinched-hard-part}
      (1 - \omega^{\nr \zeta_{\lW}^{0}} )
      \frac{
        \qlog{\zeta_{\lN}^0}{n_2' - n_1}
        \qlog{\zeta_{\lS}^0}{n_2 - n_1'}
        }{
        \qlog{\zeta_{\lW}^0}{n_2 - n_1 - 1}
        \qlog{\zeta_{\lE}^0}{n_2' - n_1'}
      }
      .
    \end{equation}
    Notice that for \(k \in \set{-\nr, -(\nr -1), \dots, \nr -1}\) and
    \(
    \omega^{\nr \zeta^1} = (1 - \omega^{\nr \zeta^0})^{-1}
    \)
    we have
    \[
      \qlog{\zeta^0}{k}
      =
      (1 - \omega^{\nr \zeta^0})^{1 - \cutoff{k}} \qlog{\zeta^0}{\modb{k}}
      =
      \omega^{\nr \zeta^1(\cutoff{k} - 1)} \qlog{\zeta^0}{\modb{k}}
    \]
    so by using
    \[
      -\zeta_{\lN}^1 - \zeta_{\lS}^1
      +\zeta_{\lW}^1 + \zeta_{\lE}^1
      =
      \alpha_{1'} - \alpha_{1}
    \]
    we see that \eqref{eq:pinched-hard-part} is the product of
    \begin{gather}
      \label{eq:pinched-cuttoff-term}
      \frac{a_{1'}}{a_{1}}
      \omega^{
        \nr\left[
          \cutoff{n_2' -  n_1} \zeta_{\lN}^1
          +
          \cutoff{n_2 -  n_1'} \zeta_{\lS}^1
          -
          (\cutoff{n_2 -  n_1-1} + 1) \zeta_{\lW}^1
          -
          \cutoff{n_2' -  n_1'} \zeta_{\lE}^1
        \right]
      }
      \intertext{and}
      \label{eq:pinched-qfac-term}
      \frac{
        \qlog{\zeta_{\lN}^0}{\modb{n_2' - n_1}}
        \qlog{\zeta_{\lS}^0}{\modb{n_2 - n_1'}}
        }{
        \qlog{\zeta_{\lW}^0}{\modb{n_2 - n_1-1}}
        \qlog{\zeta_{\lE}^0}{\modb{n_2' - n_1'}}
      }
    \end{gather}
    Because \(\modb{n} \in \set{0, \dots, \nr -1}\) it is clear that the limiting value of \eqref{eq:pinched-qfac-term} is
    \[
      \frac{
        \qp{\omega}{\modb{n_2' - n_1'}}
        \qp{\omega}{\modb{n_2 - n_1 - 1}}
        }{
        \qp{\omega}{\modb{n_2' - n_1}}
        \qp{\omega}{\modb{n_2 - n_1'}}
      },
    \]
    while \eqref{eq:pinched-cuttoff-term} requires more care.
  
    Each \(\omega^{\nr \zeta_j^1}\) has an order \(1\) pole in the limit, so \eqref{eq:pinched-cuttoff-term} (hence the matrix coefficient) has a zero of order 
    \begin{equation}
      \label{eq:matrix-coeff-zero-order}
      \cutoff{n_2 - n_1 - 1}
      +
      1
      +
      \cutoff{n_2' - n_1'}
      -
      \cutoff{n_2' - n_1}
      -
      \cutoff{n_2 - n_1'}.
    \end{equation}
    This expression is negative when
    \[
      n_2' < n_1, n_2 < n_1', n_1 < n_2 \text{, and } n_1' \le n_2',
    \]
    which is impossible.
    Therefore the matrix coefficients never have poles, which proves \ref{thm:R-mat-pinched:regular}.
    Furthermore this argument shows that \(\rmat{n_1}{n_2}{n_1'}{n_2'}\) is nonzero only when
    \begin{equation}
      \label{eq:matrix-coeff-nonzero-cond}
      \cutoff{n_2 - n_1 - 1}
      +
      \cutoff{n_2' - n_1'}
      -
      \cutoff{n_2' - n_1}
      -
      \cutoff{n_2 - n_1'}
      =
      -1.
    \end{equation}
    By using
    \[
      \cutoff{n} = 1 + \frac{n - \modb{n}}{\nr} \text{ for } n \in \set{-\nr, \dots, \nr -1}
    \]
    and
    \[
      \modb{n} = \nr - 1 - \modb{-n - 1}
    \]
    we see that \eqref{eq:matrix-coeff-nonzero-cond} holds exactly when
    \begin{equation}
      \label{eq:matrix-coeff-nonzero-modb}
      \modb{n_1 - n_2} + \modb{n_1' - n_2' -1} + \modb{n_2' - n_1} + \modb{n_2 - n_1'} = \nr -1.
    \end{equation}
    Following \cite[Section 4.1]{Murakami2018} we note that
    \[
      \modb{n_1 - n_2} + \modb{n_1' - n_2' -1} + \modb{n_2' - n_1} + \modb{n_2 - n_1'} \equiv -1 \pmod{\nr}
    \]
    which shows that \cref{eq:matrix-coeff-nonzero-modb} is true if and only if both
    \(
    \modb{n_1 - n_2} + \modb{n_1' - n_2' -1}
    \)
    and
    \(
    \modb{n_2' - n_1} + \modb{n_2 - n_1'}
    \)
    are less than \(\nr\).
    We conclude that the matrix coefficient \(\rmat{n_1}{n_2}{n_1'}{n_2'}\) is nonzero if and only if
    \[
      \cutoff{
        \modb{n_1 - n_2} + \modb{n_1' - n_2' -1}
      }
      \cutoff{
        \modb{n_2' - n_1} + \modb{n_2 - n_1'}
      }
      =
      1
    \]
    which gives the conditional expression \(\theta_{n_1 n_2}^{n_1' n_2'}\) above.
  
    To derive \(A_{n_1 n_2}^{n_1' n_2'}\), use (\ref{eq:flattening-N}--\ref{eq:flattening-E}) to write
    \begin{align*}
      &
      \cutoff{n_2' -  n_1} \zeta_{\lN}^1
      +
      \cutoff{n_2 -  n_1'} \zeta_{\lS}^1
      -
      (\cutoff{n_2 -  n_1-1} + 1) \zeta_{\lW}^1
      -
      \cutoff{n_2' -  n_1'} \zeta_{\lE}^1
      \\
      &=
      \left[
        \cutoff{n_2' -  n_1}
        +
        \cutoff{n_2 -  n_1'}
        -
        (\cutoff{n_2 -  n_1-1} + 1)
        -
        \cutoff{n_2' -  n_1'}
      \right]
      (\kappa - \gamma_{\lN})
      \\
      &\phantom{=}+
      \cutoff{n_2 -  n_1'}(- \alpha_1 - \alpha_2 + \mu_1 - \mu_2)
      \\
      &\phantom{=}-
      [\cutoff{n_2 - n_1 - 1} + 1](- \alpha_1 + \mu_1)
      -
      \cutoff{n_2' - n_1'}(-\alpha_{2'} - \mu_2)
    \end{align*}
    We just showed that \(\rmat{n_1}{n_2}{n_1'}{n_2'}\) is nonzero exactly when the coefficient of \((\kappa - \gamma_{\lN})\) above vanishes, and after using the identity
    \[
      \cutoff{n} = 1 - \cutoff{-n-1}
    \]
    we see that \eqref{eq:pinched-cuttoff-term} is equal to
    \[
      \frac{a_{1'}}{a_1}
      \left( \frac{a_1}{ m_1} \right)^{2 - \cutoff{n_1 - n_2} - \cutoff{n_2 - n_1'}}
      (a_2 m_2)^{- \cutoff{n_2 - n_1'}}
      (a_{2'} m_2)^{1- \cutoff{n_1' - n_2' - 1}}
    \]
    which is precisely \(A_{n_1 n_2}^{n_1' n_2'}\).
    This establishes \ref{thm:R-mat-pinched:formula}.
  
    Part \ref{thm:R-mat-pinched:Kashaev} follows immediately from applying the identities
    \[
      \modb{n} = \nr - 1 - \modb{-n - 1}
    \]
    and
    \[
      \qp{\omega}{n}
      \qp{\overline{\omega}}[\overline{\omega}]{\nr - 1 - n}
      =
      \nr
      \text{ for }
      0 \le n < \nr.
      \qedhere
    \]
\end{proof}

\subsection{The determinant of the \texorpdfstring{\(R\)}{R}-matrix}
Here we compute the determinant of the \(R\)-matrix, a key step in the proof of \cref{thm:RIII}.
\begin{lemma}
  \label{thm:R-mat-det}
  Let \(R\) be the \(R\)-matrix associated to a log-colored crossing \(c\) of sign \(\epsilon\).
  The determinant of the associated braiding \(\tau R\) is 
  \begin{align*}
    \det \tau R
    =
    \exp\left(-\frac{\nr}{2\pi i} I(c)\right)
    \left(
      \frac{N}{\df{0}^2}
    \right)^{\epsilon \nr^2}
    \exp\leftfun(
      2\pi i\left( \frac{\gamma_{\lW} - \gamma_{\lE}}{2} - \epsilon(\mu_{1} + \mu_{2}) + \lambda_{1} + \lambda_{2} \right)
    \rightfun)^{\nr (\nr-1)}
  \end{align*}
  where 
  \begin{align*}
    \lambda_{1} = \frac{\epsilon}{2}(\beta_{1'} - \beta_{1})
    \\
    \lambda_{2} = \frac{\epsilon}{2}(\beta_{2} - \beta_{2'})
  \end{align*}
  are the log-longitudes of each component of the crossing,
  \[
    I(c) =
      \dilr(\zeta^0_{\lN}, \zeta^1_{\lN})
      +
      \dilr(\zeta^0_{\lS}, \zeta^1_{\lS})
      -
      \dilr(\zeta^0_{\lW}, \zeta^1_{\lW})
      -
      \dilr(\zeta^0_{\lE}, \zeta^1_{\lE})
  \]
  is a sum of dilogarithms, and \(\df{0} = \exp\left( \nr^{-1} \sum_{k=0}^{\nr-1} k \log(1 - \omega^{k}) \right)\) is a constant.
\end{lemma}

\begin{proof}
  Using the factorization \eqref{eq:R-factorization-positive} of \cref{thm:R-factorization} we see that
  \[
    \det \tau R =
    \nr^{-\nr^2}
    \det \mathcal{Z}_{\lE}
    \det \mathcal{Z}_{\lW}
    (\det \mathcal{Z}_{\lN})^N
    (\det \mathcal{Z}_{\lS})^N.
  \]
  It is easy to compute the determinant of the diagonal matrices \(\mathcal{Z}_{\lE}\) and \(\mathcal{Z}_{\lW}\) using \cref{thm:qlf-product}:
  \begin{align*}
    \det \mathcal{Z}_{\lE}
    &=
    \prod_{n_1, n_2}
    \frac{
      1
    }{
      \qlf{\zeta^0_{\lE} , \zeta^1_{\lE}}{n_1 - n_2}
    }
    \\
    &=
    \left[
      \prod_{n}
      \frac{
        1
      }{
        \qlf{\zeta^0_{\lE} , \zeta^1_{\lE}}{n_1 - n_2}
      }
    \right]^{\nr}
    \\
    &=
    \omega^{\frac{1}{2} \nr^2(\nr-1) \zeta_{\lE}^1}
    \exp
    \leftfun(
      \nr
      \frac{\dilr(\zeta^0_{\lE}, \zeta_{\lE}^1)}{2\pi i}
    \rightfun)
  \end{align*}
  and similarly
  \begin{align*}
    \det \mathcal{Z}_{\lW}
    &=
    \omega^{-\frac{1}{2}\nr^2(\nr-1) (2 \zeta_{\lW}^0 + \zeta_{\lW}^1)}
    \exp
    \leftfun(
      \nr
      \frac{\dilr(\zeta^0_{\lW}, \zeta_{\lW}^1)}{2\pi i}
    \rightfun)
  \end{align*}
  Using \cref{thm:hard-determinant} to compute \(\det \mathcal{Z}_{\lN}\) and \(\det \mathcal{Z}_{\lS}\) we see that
  \begin{align*}
    \det \tau R
    =
    {}
    &
      \exp \left(
        \nr
        \frac{
          \dilr(\zeta^0_{\lW}, \zeta^1_{\lW})
          +
          \dilr(\zeta^0_{\lE}, \zeta^1_{\lE})
          -
          \dilr(\zeta^0_{\lN}, \zeta^1_{\lN})
          -
          \dilr(\zeta^0_{\lS}, \zeta^1_{\lS})
        }{
          2\pi i
        }
      \right)
    \\
    &\times
    \left(
      \frac{N}{\df{0}^2}
    \right)^{\nr^2}
    \exp
    \leftfun(
    2 \pi i
    \frac{\nr (\nr -1)}{2}
    \left(
      \zeta_{\lE}^1 - \zeta_{\lW}^1
      + 2 \zeta_{\lW}^0
      - \zeta_{\lN}^0 - \zeta_{\lS}^0
    \right)
    \rightfun)
  \end{align*}
  The claim follows from checking that
  \begin{equation*}
    \zeta_{\lE}^1 - \zeta_{\lW}^1
    + 2 \zeta_{\lW}^0
    - \zeta_{\lN}^0 - \zeta_{\lS}^0
    =
    \gamma_{\lW} - \gamma_{\lE} - 2 \epsilon(\mu_{1} + \mu_{2}) + 2 \lambda_{1} + 2 \lambda_{2}.
    \qedhere
  \end{equation*}
\end{proof}

\subsection{The Fourier transform of the \texorpdfstring{\(R\)}{R}-matrix}
\label{sec:fourier-transform-R-mat}
In our computation of the \(R\)-matrix coefficients in \cref{subsec:recurrences} we used a nonstandard basis \(\set{\vbh n \given n \in \mathbb{Z}/\nr \mathbb{Z}}\) Fourier dual to the usual highest-weight basis \(\set{\vb n \given n \in \mathbb{Z}/\nr \mathbb{Z}}\).
This turns out to give simpler relations that lead to the factorization of \cref{thm:R-factorization}.
However to relate our construction to the standard one we need to return to the highest-weight basis.
For a crossing with a geometrically degenerate coloring we explicitly recover the \(R\)-matrices defining the colored Jones polynomials and ADO invariants, generalizing a result of \textcite{Murakami2001}.

\begin{definition}
  We say a \(\chi\)-colored crossing is \defemph{\(E\)-nilpotent} if \(a_{i} = m_{i}\) for each segment \(i\) at the crossing.
  We say a log-coloring of it is \defemph{standard} if it satisfies
  \begin{equation}
    \label{eq:standard-flattening-E-nilpotent}
    \alpha_{1} = \alpha_{1'} = \mu_{1}, \alpha_{2} = \alpha_{2'} = \mu_{2}.
  \end{equation}
 the conditions of \cref{def:standard-flattening} on the log-parameters \(\beta_{i}\) and the additional relations
\end{definition}

The generator \(E \in \U_{\xi}\) acts nilpotently on every module \(\rep{\chi_{i}, \mu_{i}}\) at an \(E\)-nilpotent crossing.
When the holonomy \(\rho\) is reducible one can always choose a \(\chi\)-coloring for which every crossing is pinched and \(E\)-nilpotent \cite{McPhailSnyderVolume} and a log-coloring for which every crossing is standard, i.e.\ satisfies both \eqref{eq:standard-flattening} and \eqref{eq:standard-flattening-E-nilpotent}.

\begin{theorem}
  \label{thm:pinched-R-matrix-dual}
  At a positive crossing which is pinched, \(E\)-nilpotent, and has a standard log-coloring the coefficients of the \(R\)-matrix \eqref{eq:R-mat-positive} with respect to the bases \(\set{\vb{n}}\) are
  \begin{equation}
    \label{eq:nilpotent-R-matrix-dual}
    R_{n_1 n_2}^{n_1' n_2'}
    =
    \del{n_1 + n_2}{n_1' +  n_2'}
    \omega^{n_1'(-2 \mu_2 + n_2)}
    \frac{
      \qp{\omega^{-2\mu_2}}{n_2}
      \qp{\omega}{n_1}
      }{
      \qp{\omega^{-2\mu_2}}{n_2'}
      \qp{\omega}{n_2' - n_2}
      \qp{\omega}{n_1'}
    }
  \end{equation}
  where we assume \(n_j \in \set{0, \dots, \nr -1}\).
\end{theorem}

\begin{remark}
  When \(\alpha_j = \mu_j = -1/2 \text{ or } (\nr - 1)/2\), we obtain
  \begin{equation}
    \label{eq:colored-Jones-R-matrix}
    R_{n_1 n_2}^{n_1' n_2'}
    =
    \del{n_1 + n_2}{n_1' +  n_2'}
    \omega^{n_1'(1 + n_2)}
    \frac{
      \qp{\omega}{n_2}
      \qp{\omega}{n_1}
      }{
      \qp{\omega}{n_2'}
      \qp{\omega}{n_2' - n_2}
      \qp{\omega}{n_1'}
    }
  \end{equation}
  which is the \(R\)-matrix defining a framed version of the \(\nr\)th colored Jones polynomial at an \(\nr\)th root of unity.
  This is the \(n = m = -1\) case of \cite[eq.\@ 43]{Garoufalidis2021descendant}.
  More generally when \(\mu_j\) is a half-integer our \(R\)-matrices appear to be equivalent to those of \cite[eq.\@ 43]{Garoufalidis2021descendant}.
  This may lead to a resolution of \cite[Conjecture 3.2]{Garoufalidis2021descendant}.
\end{remark}

The proof (inspired by \cite[Section 3.2]{Garoufalidis2021descendant}) is a series of calculations involving terminating \(q\)-hypergeometric series.
We break the computation into a few lemmas.
Recall that the action of the \(R\)-matrix at a positive crossing is given by
\[
  R(\vbh{n_1 n_2})
  =
  \sum_{n_1' n_2'}
  \widehat R_{n_1 n_2}^{n_1' n_2'}
  \vbh{n_1' n_2'}
\]
where the matrix coefficients are given by \cref{eq:R-mat-positive}.
From \eqref{eq:basis-and-dual-basis} it is immediate that the action in the highest-weight basis
\[
  R(\vb{n_1 n_2})
  =
  \sum_{n_1' n_2'}
  R_{n_1 n_2}^{n_1' n_2'}
  \vb{n_1' n_2'}
\]
has matrix coefficients
\begin{equation}
  \label{eq:fourier-dual-coefficients-definition}
  R_{n_1 n_2}^{n_1' n_2'}
  =
  \frac{1}{\nr^2}
  \sum_{k_1 k_2 k_1' k_2'}
  \omega^{n_1' k_1' + n_2' k_2' - n_1 k_1 - n_2 k_2}
  \rmat{k_1}{k_2}{k_1'}{k_2'}
  .
\end{equation}
Our goal is to compute these in the pinched, \(E\)-nilpotent limit.

\begin{lemma}
  \label{thm:fourier-dual-lemma-general}
  For any crossing the matrix coefficients \eqref{eq:fourier-dual-coefficients-definition} are given by
  \begin{equation}
    \label{eq:fourier-dual-coefficients}
    \begin{aligned}
      R_{n_1 n_2}^{n_1' n_2'}
      =
      {}
      &
      -
      \nrdel{n_1 + n_2}{n_1' +  n_2'}
      \frac{
        \omega^{\Theta/2}
      }{
        \nr^{2}
      }
      \frac{
        \qlog{\nu_{2}}{\nr}
      }{
        \qlog{\zeta_{\lS}^{0} - \zeta_{\lW}^{0}}{\nr - 1}
      }
      \frac{
        \pf{\zeta_{\lN}^0}
        \pf{\zeta_{\lS}^0}
      }{
        \pf{\zeta_{\lE}^0}
        \pf{\zeta_{\lW}^0}
      }
      \\
      &\times
      \omega^{\nu_{2} + (n_{2} - 1) \zeta_{\lW}^{0} - n_{2}' (\zeta_{\lE}^{0} + 1) }
      \frac{
        \qlog{\nu_{2'}}{-n_2'}
        \qlog{-\nu_{2}}{n_2 - 1}
      }{
        \qlog{-\zeta + \nu_{2'} - 1}{- n_2'}
        \qlog{\zeta - \nu_{2} - 1}{n_2 }
      }
      \\
      &\times
      \fstack{ \zeta - \nu_{2'} + n_{2}' }{ \zeta - \nu_{2} -1 + n_{2} }{ \nu_{1'} - n_{1}' }
      \fstack{\zeta_{\lN}^0 }{\zeta_{\lE}^0}{- \nu_{2'}}
      \fstack{\zeta_{\lS}^0 }{\zeta_{\lW}^0}{\nu_{2}}
    \end{aligned}
\end{equation}
where \(\Theta\) is as in \cref{thm:R-mat-phi-form},
\(\zeta = \zeta_{\lE}^{0} - \zeta_{\lN}^{0} = \zeta_{\lS} - \zeta_{\lW}^{0}\), 
and 
\[
  \fstack{\alpha}{\beta}{\gamma} \defeq
  \sum_{n = 0}^{\nr -1}
  \frac{
    \qlog{\alpha}{n}
    }{
    \qlog{\beta}{n}
  }
  \omega^{n \gamma}
\]
is the function studied by \textcite{Kashaev1993}, discussed here in \cref{sec:fusion-identities}.
\end{lemma}

\begin{proof}
Use \cref{thm:R-mat-phi-form} to write
\[
  \begin{aligned}
    \rmat{k_1}{k_2}{k_1'}{k_2'}
    &=
    \frac{
      \omega^{\Theta/2}
      \omega^{-\nr\zeta_{\lW}^0}
    }{
      \nr
    }
    (1 - \omega^{\nr \zeta_{\lW}^0})
    \\
    &\phantom{=}\times
    \omega^{
      k_1(\alpha_1-\mu_1 - 1)
      +
      k_2(\alpha_2 + \mu_2 + 1)
      -
      k_1' (\alpha_{1'} - \mu_{1})
      -
      k_2' (\alpha_{2'} + \mu_2)
    }
    \\
    &\phantom{=}\times
    \frac{
      \pf{\zeta_{\lN}^0 + k_2' - k_1}
      \pf{\zeta_{\lS}^0 + k_2 - k_1'}
      }{
        \pf{\zeta_{\lE}^0 + k_2' - k_1'}
        \pf{\zeta_{\lW}^0 + k_2 - k_1 - 1}
    }
  \end{aligned}
\]
Setting
\begin{align*}
  \nu_{1}
    &=
    \alpha_1 - \mu_1
    &
    \nu_{1'}
    &=
    \alpha_{1'} - \mu_1
    \\
    \nu_{2}
    &=
    \alpha_2 + \mu_2
    &
    \nu_{2'}
    &=
    \alpha_{2'} + \mu_2
\end{align*}
and changing summation variables to \(k_2 = k_1 + t\), \(k_1' = k_1 - k\), \(k_2' = k_1 - k + t'\) shows that
\begin{equation}
  \begin{aligned}
    R_{n_1 n_2}^{n_1' n_2'}
    =
    &
    {}
    \frac{
      \omega^{\Theta/2}
      \omega^{-\nr\zeta_{\lW}^0}
    }{
      \nr^{3}
    }
    (1 - \omega^{\nr \zeta_{\lW}^0})
    \sum_{k_1}
    \omega^{k_1(n_1' + n_2' -n_1 -n_2)}
    \\
    &\times
    \sum_{k, t, t'}
    \omega^{k(\nu_{1'} + \nu_{2'} - n_1' - n_2')}
    \omega^{t(\nu_2 + 1 - n_2)}
    \omega^{t'(n_2' - \nu_{2'})}
    \frac{
      \pf{\zeta_{\lN}^0 + t' - k}
      \pf{\zeta_{\lS}^0 + t + k}
      }{
      \pf{\zeta_{\lE}^0 + t'}
      \pf{\zeta_{\lW}^0 - 1 + t}
    }
  \end{aligned}
\end{equation}
We can apply the identity
\[
  \sum_{k_1}
  \omega^{k_1(n_1' + n_2' -n_1 -n_2)}
  =
  \nr \nrdel{n_1 + n_2}{n_1' +  n_2'}
\]
and the relation \eqref{eq:pf-recurrence} to see that
\begin{equation}
  \begin{aligned}
    R_{n_1 n_2}^{n_1' n_2'}
    =
    &
    {}
    \frac{
      \omega^{\Theta/2}
      \omega^{-\nr\zeta_{\lW}^0}
    }{
      \nr^{2}
    }
    (1 - \omega^{\nr\zeta_{\lW}^0})
    \nrdel{n_1 + n_2}{n_1' +  n_2'}
    \\
    &\times
    \sum_{k}
    \left\{
    \begin{aligned}
      &
      \omega^{k(\nu_{1} + \nu_{2} - n_1 - n_2)}
      \frac{
        \pf{\zeta_{\lN}^0 - k}
        \pf{\zeta_{\lS}^0 + k}
        }{
          \pf{\zeta_{\lE}^0}
          \pf{\zeta_{\lW}^0 - 1 }
      }
      \\
      &\times
      \sum_{t'}
      \omega^{t'(n_2' - \nu_{2'})}
      \frac{
        \qlog{\zeta_{\lN}^0 - k}{t'}
        }{
        \qlog{\zeta_{\lE}^0}{t'}
      }
      \\
      &\times
      \sum_{t}
      \omega^{t(\nu_2 + 1 - n_2)}
      \frac{
        \qlog{\zeta_{\lS}^0 + k}{t}
        }{
        \qlog{\zeta_{\lW}^0 - 1}{t}
      }
    \end{aligned}
    \right.
  \end{aligned}
\end{equation}

The sums over \(t\) and \(t'\) can be written in terms of the function \(\fstacksmall{\alpha}{\beta}{\gamma}\) of \cref{sec:fusion-identities}.
As
\[
  \fstack{\zeta_{\lN}^0 - k}{\zeta_{\lE}^0}{n_2' - \nu_{2'}}
  =
  \fstack{\zeta_{\lN}^0 }{\zeta_{\lE}^0}{- \nu_{2'}}
  \frac{
    \qlog{\zeta_{\lN}^0 - \zeta_{\lE}^0 - 1}{-k}
    \qlog{\nu_{2'}}{-n_2'}
  }{
    \omega^{n_2'(\zeta_{\lE}^0 + 1)}
    \qlog{\zeta_{\lN}^0}{-k}
    \qlog{\zeta_{\lN}^0 - \zeta_{\lE}^0 + \nu_{2'} - 1}{-k - n_2'}
  }
\]
and
\[
  \fstack{\zeta_{\lS}^0 + k}{\zeta_{\lW}^0 - 1}{\nu_2 - n_2 + 1}
  =
  \fstack{\zeta_{\lS}^0 }{\zeta_{\lW}^0}{\nu_{2}}
  \frac{
    \qlog{\zeta_{\lS}^0 - \zeta_{\lW}^0 - 1}{k + 1}
    \qlog{\zeta_{\lW}^{0} - 1}{-1}
    \qlog{-\nu_{2}}{n_2 - 1}
  }{
    \omega^{- \nu_{2} + (1 - n_2)\zeta_{\lW}^0}
    \qlog{\zeta_{\lS}^0}{k}
    \qlog{\zeta_{\lS}^0 - \zeta_{\lW}^0 - \nu_{2} - 1}{k + n_2}
  }
\]
we can substitute and apply some algebraic manipulations to get
  \begin{equation}
    \begin{aligned}
      R_{n_1 n_2}^{n_1' n_2'}
      =
      {}
      &
      \nrdel{n_1 + n_2}{n_1' +  n_2'}
      \frac{
        \omega^{\Theta/2}
      }{
        \nr^{2}
      }
      \frac{
        (1 - \omega^{\nr\zeta_{\lW}^0})
      }{
        \omega^{\nr \zeta_{\lW}^0}
        (1 - \omega^{\zeta_{\lS}^{0} - \zeta_{\lW}^0})
      }
      \frac{
        \pf{\zeta_{\lN}^0}
        \pf{\zeta_{\lS}^0}
      }{
        \pf{\zeta_{\lE}^0}
        \pf{\zeta_{\lW}^0}
      }
      \fstack{\zeta_{\lN}^0 }{\zeta_{\lE}^0}{- \nu_{2'}}
      \fstack{\zeta_{\lS}^0 }{\zeta_{\lW}^0}{\nu_{2}}
      \\
      &\times
      \omega^{\nu_{2} + (n_{2} - 1) \zeta_{\lW}^{0} - n_{2}' (\zeta_{\lE}^{0} + 1) }
      \qlog{\nu_{2'}}{-n_2'}
      \qlog{-\nu_{2}}{n_2 - 1}
      \\
      &\times
      \sum_{k}
      \frac{
        \omega^{k(\nu_{1} + \nu_{2} - n_1 - n_2)}
        \qlog{\zeta_{\lN}^0 - \zeta_{\lE}^0- 1}{-k}
        \qlog{\zeta_{\lS}^0 - \zeta_{\lW}^0}{k}
      }{
        \qlog{\zeta_{\lN}^0 - \zeta_{\lE}^0 + \nu_{2'} - 1}{-k - n_2'}
        \qlog{\zeta_{\lS}^0 - \zeta_{\lW}^0 - \nu_{2} - 1}{k + n_2 }
      }
    \end{aligned}
  \end{equation}
  The constant can be further simplified by applying \cref{eq:b-transf-positive} to write
  \begin{align*}
    \frac{
      (1 - \omega^{\nr\zeta_{\lW}^0})
    }{
      \omega^{\nr \zeta_{\lW}^0}
      (1 - \omega^{\zeta_{\lS}^{0} - \zeta_{\lW}^0})
    }
    &=
    \frac{
      (1 - \omega^{\nr\zeta_{\lW}^0})
    }{
      \omega^{\nr \zeta_{\lW}^0}
      (1 - \omega^{\nr(\zeta_{\lS}^{0} - \zeta_{\lW}^0)})
      \qlog{\zeta_{\lS}^{0} - \zeta_{\lW}^{0}}{\nr - 1}
    }
    \\
    &=
    \frac{
      1
      }{
        (a_{2} m_{2} - 1)
      \qlog{\zeta_{\lS}^{0} - \zeta_{\lW}^{0}}{\nr - 1}
    }
    \\
    &=
    -
    \frac{
      \qlog{\nu_{2}}{\nr}
      }{
      \qlog{\zeta_{\lS}^{0} - \zeta_{\lW}^{0}}{\nr - 1}
    }
  \end{align*}
  Now
  \[
    \qlog{\zeta}{m + n} = \qlog{\zeta}{m}\qlog{\zeta+m}{n}
  \]
  and the inversion relation
  \[
    \qlog{\alpha}{-k}
    =
    \omega^{k \alpha} (-1)^{k} \omega^{k(k+1)/2} \qlog{-\alpha}{k}
  \]
  shows that
  \begin{align*}
    &\frac{
      \qlog{\zeta_{\lN}^0 - \zeta_{\lE}^0- 1}{-k}
    }{
      \qlog{\zeta_{\lN}^0 - \zeta_{\lE}^0 + \nu_{2'} - 1}{-k - n_2'}
    }
    \\
    &=
    \frac{
      \qlog{\zeta_{\lN}^0 - \zeta_{\lE}^0- 1}{-k}
    }{
      \qlog{\zeta_{\lN}^0 - \zeta_{\lE}^0 + \nu_{2'} - 1}{- n_2'}
      \qlog{\zeta_{\lN}^0 - \zeta_{\lE}^0 + \nu_{2'} - 1 - n_{2}'}{-k }
    }
    \\
    &=
    \omega^{k(n_{2}' - \nu_{2'})}
    \frac{
      \qlog{\zeta_{\lE}^0 - \zeta_{\lN}^0 - \nu_{2'} + n_{2}'}{k }
    }{
      \qlog{\zeta_{\lN}^0 - \zeta_{\lE}^0 + \nu_{2'} - 1}{- n_2'}
      \qlog{\zeta_{\lE}^0 - \zeta_{\lN}^0}{k}
    }
  \end{align*}
  so using \(\nu_{1} + \nu_{2} - \nu_{2'} = \nu_{1'}\) and \(\zeta_{\lE}^{0} + \zeta_{\lW}^{0} = \zeta_{\lN}^0 + \zeta_{\lS}^0\) we get
  \begin{align*}
      R_{n_1 n_2}^{n_1' n_2'}
      =
      {}
      &
      -
      \nrdel{n_1 + n_2}{n_1' +  n_2'}
      \frac{
        \omega^{\Theta/2}
      }{
        \nr^{2}
      }
      \frac{
        \qlog{\nu_{2}}{\nr}
      }{
        \qlog{\zeta_{\lS}^{0} - \zeta_{\lW}^{0}}{\nr - 1}
      }
      \frac{
        \pf{\zeta_{\lN}^0}
        \pf{\zeta_{\lS}^0}
      }{
        \pf{\zeta_{\lE}^0}
        \pf{\zeta_{\lW}^0}
      }
      \fstack{\zeta_{\lN}^0 }{\zeta_{\lE}^0}{- \nu_{2'}}
      \fstack{\zeta_{\lS}^0 }{\zeta_{\lW}^0}{\nu_{2}}
      \\
      &\times
      \omega^{\nu_{2} + (n_{2} - 1) \zeta_{\lW}^{0} - n_{2}' (\zeta_{\lE}^{0} + 1) }
      \frac{
        \qlog{\nu_{2'}}{-n_2'}
        \qlog{-\nu_{2}}{n_2 - 1}
        }{
        \qlog{\zeta_{\lN}^0 - \zeta_{\lE}^0 + \nu_{2'} - 1}{- n_2'}
        \qlog{\zeta_{\lS}^0 - \zeta_{\lW}^0 - \nu_{2} - 1}{n_2 }
      }
      \\
      &\times
      \sum_{k}
        \omega^{k(\nu_{1'} - n_1' )}
      \frac{
        \qlog{\zeta_{\lE}^0 - \zeta_{\lN}^0 - \nu_{2'} + n_{2}'}{k }
      }{
        \qlog{\zeta_{\lS}^0 - \zeta_{\lW}^0 - \nu_{2} - 1+ n_2 }{k }
      }
  \end{align*}
  as claimed.
\end{proof}

\begin{lemma}
  \label{thm:fourier-dual-lemma-pinched}
  Suppose that our crossing is pinched and has a standard flattening.
  Then the matrix coefficients \eqref{eq:fourier-dual-coefficients-definition} are given by 
  \begin{equation}
    \label{eq:fourier-dual-coefficients-pinched}
    R_{n_{1} n_{2}}^{n_{1}' n_{2}'}
    =
    \frac{
      \nrdel{n_1 + n_2}{n_1' +  n_2'}
      }{
      \nr 
    }
    \frac{
      1- \omega^{-\nr \nu_{2'}}
      }{
      1- \omega^{-\nu_{2'} + n_{2}'}
    }
    \fstack{- \nu_{2'} + n_{2}' }{- \nu_{2} + n_{2} - 1}{\nu_{1'} - n_{1}' }
  \end{equation}
\end{lemma}

\begin{proof}
  We are taking the limit where all \(\zeta_{i}^{0} \to 0\), which gives
  \begin{equation*}
    \begin{aligned}
      R_{n_1 n_2}^{n_1' n_2'}
      =
      {}
      &
      -
      \nrdel{n_1 + n_2}{n_1' +  n_2'}
      \frac{
        1
      }{
        \nr^{2}
      }
      \frac{
        \qlog{\nu_{2}}{\nr}
      }{
        \nr^{-1}
      }
      \\
      &\times
      \omega^{\nu_{2} }
      \frac{
        \qlog{\nu_{2'}}{-n_2'}
        \qlog{-\nu_{2}}{n_2 - 1}
      }{
        \qlog{\nu_{2'} - 1}{- n_2'}
        \qlog{- \nu_{2} - 1}{n_2 }
      }
      \\
      &\times
      \fstack{- \nu_{2'} + n_{2}' }{ - \nu_{2} -1 + n_{2} }{ \nu_{1'} - n_{1}' }
      \frac{
        1 - \omega^{-\nr\nu_{2}'}
        }{
        1 - \omega^{-\nu_{2}'}
      }
      \frac{
        1 - \omega^{\nr\nu_{2}}
        }{
        1 - \omega^{\nu_{2}}
      }
    \end{aligned}
  \end{equation*}
  because 
  \[
    \fstack{\zeta_{\lS}^0 }{\zeta_{\lW}^0}{\nu_{2}}
    \to
    \sum_{k=0}^{\nr - 1}
    \frac{
      \qlog{0}{k}
      }{
      \qlog{0}{k}
    }
    \omega^{\nu_{2} k}
    =
    \frac{
      1 - \omega^{\nr\nu_{2}}
    }{
      1 - \omega^{\nu_{2}}
    }
  \]
  and similarly for \(-\nu_{2'}\).
  Simplifying gives the claim.
\end{proof}

\begin{proof}[Proof of \cref{thm:pinched-R-matrix-dual}]
  Finally we take the limit of the matrix coefficients \eqref{eq:fourier-dual-coefficients-pinched} where the crossing is \(E\)-nilpotent.
  Since we have a standard flattening \eqref{eq:standard-flattening-E-nilpotent} gives \(\nu_{2} = \nu_{2'}  = 2 \mu_{2}\) and \(\nu_{1'} = 0\), so using \cref{eq:fusion-identity-ii} the coefficients become
  \begin{align*}
    R_{n_{1} n_{2}}^{n_{1}' n_{2}'}
    =
    \nrdel{n_1 + n_2}{n_1' +  n_2'}
    \frac{
      1- \omega^{-\nu_{2} + n_{2}}
      }{
      1- \omega^{-\nu_{2} + n_{2}'}
    }
    \frac{
      \omega^{ \modb{n_{1}'}(-\nu_{2} + n_{2})}
    }{
      \qlog{-\nu_{2} + n_{2}}{ \modb{n_{2}' - n_{2}} }
    }
    \frac{
      \qp{\omega}{\modb{n_2' - n_2} + \modb{n_1'}}
    }{
      \qp{\omega}{\modb{n_2' - n_2}}
      \qp{\omega}{\modb{n_1'}}
    }
  \end{align*}
  Because our indices lie in \(\set{0, \dots, \nr -1}\) we can assume that \(\modb{n_1'} = n_1'\).
  In addition, \(\modb{n_2' - n_2} = \modb{n_1 - n_1'}\) because \(n_1 + n_2 \equiv n_1' + n_2' \pmod \nr\).
  Furthermore, if \(n_1 < n_1'\), then
  \[
    \modb{n_1 - n_1'} + n_1' = \nr + n_1 \ge \nr - 1
  \]
  so
  \[
    \qp{\omega}{\modb{n_2' - n_2} + \modb{n_1'}} = 0
    \text{ if }
    n_1 < n_1'.
  \]

  For \(n_1, n_2, n_1', n_2' \in \set{0, \dots, \nr -1}\) the conditions \(n_1 + n_2 \equiv n_1' + n_2' \pmod \nr\) and \(n_1 \ge n_1'\) are equivalent to \(n_1 + n_2 = n_1' + n_2'\) and \(n_2' \ge n_2\), and we conclude that
  \begin{align*}
    R_{n_{1} n_{2}}^{n_{1}' n_{2}'}
    =
    \del{n_1 + n_2}{n_1' +  n_2'}
    \frac{
      1- \omega^{-\nu_{2} + n_{2}}
      }{
      1- \omega^{-\nu_{2} + n_{2}'}
    }
    \frac{
      \omega^{ n_{1}'(n_{2} - \nu_{2})}
    }{
      \qlog{-\nu_{2} + n_{2}}{ n_{2}' - n_{2} }
    }
    \frac{
      \qp{\omega}{n_{1}}
    }{
      \qp{\omega}{n_2' - n_2}
      \qp{\omega}{n_1'}
    }
  \end{align*}
  because if \(n_2' < n_2\), we can use
  \[
    \frac{1}{\qp{\omega}{-k}} = (1 - 1) (1 - \omega^{-1}) \cdots (1 - \omega^{-(k-1)}) = 0 \text{ for } k > 0.
  \]
  Finally, re-writing
  \begin{align*}
    \frac{
      1 - \omega^{-\nu_2 + n_2}
    }{
      1 - \omega^{-\nu_2 +  n_2'}
    }
    \frac{
      1
    }{
      \qlog{-\nu_2 + n_2}{n_2' - n_2}
    }
    =
    \frac{
      \qp{\omega^{-\nu_2}}{n_2}
    }{
      \qp{\omega^{-\nu_2}}{n_2'}
    }
  \end{align*}
  gives the claim.
\end{proof}

\appendix
\section{Quantum dilogarithms}
\label{sec:quantum-dilogarithms}
This appendix contains some facts about quantum dilogarithms and \(q\)-series used in the earlier parts of the paper (of which it is logically independent).
Most of them were previously known; they are collected here for convenience and to ensure they match our conventions.

Recall that \(\nr \ge 2\) is an integer and
\[
  \omega \defeq \exp(2\pi i /\nr) \text{ and } \omega^{x} \defeq \exp(2\pi i x/\nr) \text{ for } x \in \CC.
\]
To define the complex logarithm we take a branch cut along the negative real axis and choose \(\Im \log \in (-\pi, \pi]\).

\subsection{Basic definitions}
\label{sec:qlog-definitions}
For a nonzero complex parameter \(\mathsf b\), \defemph{Faddeev's noncompact quantum dilogarithm} is defined by
\begin{equation}
  \Phi_{\mathsf b}(z) \defeq
  \exp
  \int_{\mathbb{R} + i\epsilon}
  \frac{
    \exp(-2izw)
  }{
    4 \sinh(w \mathsf b ) \sinh(w/\mathsf b)
  }
  \frac{dw}{w}
\end{equation}
for \(|\Im z| < |\Im c_{\mathsf b}|\), where
\begin{equation}
  c_{\mathsf b} \defeq \frac{i}{2} \left( \mathsf b + \mathsf b^{-1}\right).
\end{equation}
\(\Phi_{\mathsf b}\) extends to a meromorphic function on \(\CC\) with an essential singularity at infinity.

We consider the case where \(\mathsf b = \sqrt \nr\) and the normalization
\begin{equation}
  \pf{\zeta}
  =
  \pf[\nr]{\zeta}
  \defeq
  \Phi_{\sqrt \nr}\left(i \frac{\zeta}{\sqrt \nr} - c_{\sqrt N} + \frac{i}{\sqrt \nr} \right)
\end{equation}
One reason for this is to get a convenient relation with the \(q\)-factorial.
The standard recurrence identities \cite[Section 6]{Faddeev2001} give
\begin{equation}
  \label{eq:pf-recurrence}
  \pf{\zeta + k} = \pf{\zeta} (1 - \omega^{\zeta + 1})^{-1} (1 - \omega^{\zeta + 2})^{-1} \cdots (1 - \omega^{\zeta + k})^{-1}
  .
\end{equation}
This can be conveniently written as
\begin{equation*}
  \pf{\zeta + k} = \pf{\zeta} \qlog{\zeta}{k} \text{ for all } k \in \ZZ
\end{equation*}
in terms of the function%
\note{
  It is more common to write \(\qlog{\omega^{\zeta}}{k}\) for what we denote \(\qlog{\zeta}{k}\), but we prefer the logarithmic notation.
}
\(\qlog{\zeta}{k}\) defined for \(\zeta \in \CC, k \in \ZZ\) by
\begin{equation}
  \label{eq:qlog-def}
  \qlog{\zeta}{k} = \frac{\qlog{\zeta}{k-1}}{1 - \omega^{\zeta + k}}, \quad \qlog{\zeta}{0} = 1.
\end{equation}
The function \(\qlogname\) is sometimes called a cyclic quantum dilogarithm \cite[Section 3]{Faddeev1994}.
Note that 
\[
  \qlog{\zeta}{-k} = (1 - \omega^{\zeta}) (1 - \omega^{\zeta -1}) \cdots (1 - \omega^{\zeta - (k-1)})
  \text{ for } k > 0
\]
and in general \(\qlogname\) can be written using the \(q\)-Pochhammer symbol
\begin{equation}
  \label{eq:qp-def}
  \qp{a}[q]{k}
  \defeq
  \begin{cases}
    (1 - a)(1-aq) \cdots (1 - a q^{k-1})
    &
    k > 0
    \\
    1
    &
    k = 0
    \\
    \left[
      (1-aq^{-1}) \cdots (1 - a q^{-k})
    \right]^{-1}
    &
    k < 0
  \end{cases}
\end{equation}
as 
\begin{equation}
  \qlog{\zeta}{k}
  =
  \frac{1}{\qp{\omega^{\zeta+1}}{k}}.
\end{equation}

Now suppose \(\zeta^0, \zeta^1\) are complex parameters with
\begin{equation}
  \label{eq:flattening-condition-tet}
  \exp(2\pi i \zeta^1) = \frac{1}{1 - \exp(2\pi i \zeta^0)}
\end{equation}
equivalently
\[
  \exp(2\pi i \zeta^0) + \exp(-2\pi i \zeta^1) = 1.
\]
\begin{definition}
  \label{def:qlf}
  For \(\zeta^0, \zeta^1 \in \CC\) satisfying \eqref{eq:flattening-condition-tet} and  \(n \in \ZZ\)
  the \defemph{cyclic quantum dilogarithm} is the function
  \begin{equation}
    \label{eq:qlf-def}
    \qlf{\zeta^0, \zeta^1}{n} \defeq
    \omega^{-\zeta^0\zeta^1/2} 
    \omega^{-n \zeta^1}
    \pf{\zeta^0 + n}.
  \end{equation}
\end{definition}

\begin{proposition}
  It satisfies the recurrence relation
  \begin{equation}
      \qlf{\zeta^0, \zeta^1}{n}
      =
      \qlf{\zeta^0, \zeta^1}{0}
      \omega^{-n \zeta^1}
      \qlog{\zeta^0}{n}
  \end{equation}
  and its integer argument is periodic modulo \(\nr\) :
  \begin{equation}
    \qlf{\zeta^0, \zeta^1}{n + \nr}
    = 
    \qlf{\zeta^0, \zeta^1}{n}.
  \end{equation}
  Furthermore
  \begin{align}
    \label{eq:qlf-shift-0}
    \qlf{\zeta^0 + k, \zeta^1}{n}
    &=
    \omega^{k \zeta^1/2}
    \qlf{\zeta^0 , \zeta^1}{n + k}
    \\
    \label{eq:qlf-shift-1}
    \qlf{\zeta^0 , \zeta^1 + k}{n}
    &=
    \omega^{-k \zeta^0/2 - nk}
    \qlf{\zeta^0 , \zeta^1}{n}
  \end{align}
\end{proposition}
\begin{proof}
  The recurrence relation follows from \eqref{eq:pf-recurrence}, while the periodicity is a consequence of \eqref{eq:flattening-condition-tet} and
  \begin{equation}
    \label{eq:factorization-B}
    (1 - \omega^{\zeta^0})(1 - \omega^{\zeta^0 + 1}) \cdots (1 - \omega^{\zeta^0 + \nr -1}) = 1 - \omega^{\nr \zeta^0}
    =
    \omega^{-\nr \zeta^1}
    =
    e^{-2\pi i \zeta^1}.
  \end{equation}
  The last two relations are obvious.
\end{proof}

The function is called a cyclic quantum dilogarithm in analogy with the classical lifted dilogarithm
\begin{equation}
  \label{eq:dilr-classical}
  \dilr(\zeta^0, \zeta^1)
  \defeq
  \frac{\dil(e^{2\pi i \zeta^0})}{2\pi i} + \pi i \zeta^0 \zeta^1 + \zeta^0 \log(1 - e^{2\pi i \zeta^0})
\end{equation}
where
\begin{equation}
  \dil(z)
  \defeq
  \int_{0}^{z} - \frac{\log(1 -t)}{t}dt
\end{equation}
is the usual dilogarithm and \(\zeta^0, \zeta^1\) satisfy \eqref{eq:flattening-condition-tet}.
An ideal tetrahedron with shape parameter \(e^{2\pi i \zeta^0}\) and flattening \((\zeta^0, \zeta^1)\) has complex Chern-Simons invariant \(\dilr(\zeta^0, \zeta^1) - i \pi / 12\); see \cite{McPhailSnyderVolume} for more details.
It satisfies an inversion relation:

\begin{lemma}
  \label{thm:dilog-inversion-classical}
  \[
    \exp \leftfun( \dilr(\zeta^0, \zeta^1) + \dilr(-\zeta^1, -\zeta^0) \rightfun) = e^{-\pi i /12}
    .
    \qedhere
  \]
\end{lemma}

\begin{proof}
  Using the functional equation \cite{Zagier2007}   
  \[
    \operatorname{Li}_{2}(1 - z)
    =
    -\operatorname{Li}_{2}(z)
    + \frac{\pi^{2}}{6}
    - \log(z) \log(1-z)
  \]
  we have
  \begin{align*}
  &2 \pi i[
    \dilr(\zeta^0, \zeta^1) + \dilr(-\zeta^1, -\zeta^0)
  ]
  \\
  &=
  \operatorname{Li}_{2}(z)
  +
  \operatorname{Li}_{2}(1-z)
  +
  (2 \pi i)^{2} \zeta^{0} \zeta^{1}
  +
  2 \pi i \zeta^{0} \log(1 - z)
  -
  2 \pi i \zeta^{1} \log z
  \\
  &=
  \frac{ \pi^{2} }{ 6 }
  -
  \log(z) \log(1-z) 
  +
  (2 \pi i)^{2} \zeta^{0} \zeta^{1}
  +
  2 \pi i \zeta^{0} \log(1 - z)
  -
  2 \pi i \zeta^{1} \log z
  \\
  &=
  \frac{ \pi^{2} }{ 6 }
  +
  [2 \pi i \zeta^{0} - \log (z) ]
  [2 \pi i \zeta^{1} + \log(1-z)]
  \\
  &=
  \frac{ \pi^{2} }{ 6 }
  +
  (2 \pi i)^{2} k^{0} k^{1}
  \end{align*}
  for some integers \(k^{0}, k^{1}\).
\end{proof}

\textcite{Garoufalidis2014} gave an explicit relationship between quantum and classical dilogarithms, which in our notation is
\begin{theorem}[\protect{\cite[eq.\@ (25)]{Garoufalidis2014}}]
  \label{thm:exact-qlf-value}
  The exact value is given by
  \begin{equation}
    \label{eq:exact-qlf-value}
    \qlf{\zeta^0, \zeta^1}{0}
    =
    \exp\leftfun(
    - \frac{\dilr(\zeta^0, \zeta^1)}{ \nr}
    \rightfun)
    \frac{1 - \omega^{\nr \zeta^0}}{1 - \omega^{\zeta^0}}
    \df{\zeta^0}^{-1}
  \end{equation}
  where
  \begin{equation}
    \label{eq:D-def}
    \df{\zeta}
    = \prod_{k=1}^{\nr -1} (1 - \omega^{\zeta + k})^{k/\nr}
    \defeq
    \exp \leftfun( \frac{1}{\nr} \sum_{k=1}^{\nr-1} k \log(1 - \omega^{\zeta + k}) \rightfun)
    .
  \end{equation}
\end{theorem}
The function \(\df{\zeta}^{\nr}\) satisfies a recurrence relation similar to \(\pf{\zeta}\):
\begin{equation}
  \label{eq:D-recurrence}
  \frac{
    \df{\zeta+1}^{\nr}
    }{
    \df{\zeta}^{\nr}
  }
  =
  \frac{
    (1 - \omega^{\zeta})^{\nr}
    }{
    1 - \omega^{\nr \zeta}
  }
\end{equation}

\begin{lemma}
  \label{thm:qlf-product}
  \begin{equation}
    \label{eq:qlf-product}
    \prod_{k=0}^{\nr -1} \qlf{\zeta^0, \zeta^1}{k}
    =
    \omega^{-\nr(\nr-1) \zeta^1/2}
    \exp
    \leftfun(
    - \dilr(\zeta^0, \zeta^1)
    \rightfun)
  \end{equation}
\end{lemma}
\begin{proof}
  It's a bit simpler to compute the inverse of this product.
  First observe that
  \[
    \prod_{k=1}^{\nr-1} (1 - \omega^{\zeta^0 + 1})\cdots(1 - \omega^{\zeta^0 + k})
    =
    \prod_{k=1}^{\nr-1} (1 - \omega^{\zeta^0 + 1})^{\nr - k}
  \]
  Because
  \[
    \df{\zeta^0}^{\nr}
    =
    \prod_{k=1}^{\nr-1} (1 - \omega^{\zeta^0 + k})^{k}
  \]
  we see that
  \begin{align*}
    &\left( \frac{1 - \omega^{\zeta^0} }{1 - z^0} \right)^{\nr}
    \df{\zeta^0}^{\nr}
    \prod_{k=1}^{\nr-1} (1 - \omega^{\zeta^0 + 1})\cdots(1 - \omega^{\zeta^0 + k})
    \\
    &=
    \left( \frac{1 - \omega^{\zeta^0} }{1 - z^0} \right)^{\nr}
    \prod_{k=1}^{\nr-1} (1 - \omega^{\zeta^0 + k})^{\nr}
    \\
    &=
    (1 - z^0)^{-\nr}
    \prod_{k=1}^{\nr} (1 - \omega^{\zeta^0 + k})^{\nr}
    \\
    &=
    1.
  \end{align*}
  Similarly
  \[
    \prod_{k=1}^{\nr -1} \omega^{k \zeta^1}
    =
    \omega^{\zeta^1\nr(\nr-1)/2}
    =
    \exp\leftfun( 2 \pi i \frac{(\nr -1)}{2}  \zeta^1 \rightfun)
  \]
  and combining this with \cref{thm:exact-qlf-value} gives the result.
\end{proof}

\subsection{Fusion identities}
\label{sec:fusion-identities}
Suppose \(\alpha, \beta, \gamma \in \CC\) satisfy
\begin{equation}
  \frac{
    1 - \omega^{\nr \alpha}
    }{
    1- \omega^{\nr \beta}
  }
  =
  \omega^{\nr \gamma}
\end{equation}
and consider the sum
\begin{equation}
  \label{eq:f-sum-def}
  \fstack{\alpha}{\beta}{\gamma} 
  \defeq
  \sum_{k}
  \frac{
    \qlog{\alpha}{k}
  }{
    \qlog{\beta}{k}
  }
  \omega^{k\gamma}
\end{equation}
It clearly depends only on the parameters \(\alpha, \beta, \gamma \in \CC\) modulo \(\nr \ZZ\).
By \cite[eq.\@ A.14]{Kashaev1993} we have
\begin{equation}
  \label{eq:fusion-identity-i}
  \fstack{\alpha + k}{\beta + l}{\gamma + m} 
  =
  \fstack{\alpha}{\beta}{\gamma} 
  \frac{
    \qlog{\alpha - \beta -1}{k-l}
    \qlog{\beta}{l}
    \qlog{-\gamma}{-m}
  }{
    \omega^{l (\gamma + m)}
    \omega^{m(\beta + 1)}
    \qlog{\alpha}{k}
    \qlog{\alpha - \beta - \gamma - 1}{k - l - m}
  }
\end{equation}
when \(\gamma \not \in \ZZ\).
In the special case where \(\gamma\) is an integer, so \(\alpha \equiv \beta \pmod {\ZZ}\) this becomes
\begin{equation}
  \label{eq:fusion-identity-ii}
  \fstack{\alpha + k}{\alpha + l-1}{m} 
  =
  \nr
  \frac{
    1 - \omega^{\alpha + l}
  }{
    1 - \omega^{\nr \alpha}
  }
  \frac{
    \omega^{\modb{-m}(\alpha + l)}
  }{
    \qlog{\alpha + l}{\modb{k - l}}
  }
  \frac{
    \qp{\omega}{\modb{k - l} + \modb{-m}}
  }{
    \qp{\omega}{\modb{k - l}}
    \qp{\omega}{\modb{-m}}
  }
\end{equation}
as in \cite[eq.\@ 40]{Garoufalidis2021descendant}.
Here \(\modb{k}\) is the unique integer with
\[
  0 \le \modb{k} < \nr
  \text{ and }
  \modb{k} \equiv k \pmod \nr
\]
and \(\qp{\omega}[\omega]{k}\) is the \(q\)-Pochhammer symbol defined in \cref{eq:qp-def}.
Another useful identity comes from taking the limit of \eqref{eq:f-sum-def} as \(\alpha \to -\infty\) and using \cite[eq.\@ C.7]{Kashaev1993} to obtain
\begin{equation}
  \label{eq:f-Nth-power}
  \left[
    \sum_{k=0}^{\nr-1}
    \frac{\omega^{k \gamma}}{\qlog{\beta}{k}}
  \right]^{\nr}
  =
  \left[
    \omega^{(\nr-1)\beta}
    \frac{
      \df{0}
    }{
      \df{-\gamma + 1}
      \df{\beta + 1}
  }
  \right]^{\nr}
  .
\end{equation}

\subsection{Fourier transform identities}
For a few of our computations we need to understand how the operators appearing in the factorization of the \(R\)-matrix in \cref{thm:R-factorization} act in the basis \(\set{\vb n}\) Fourier dual to \(\set{\vbh n}\).
In particular, we want to compute the determinant of the circulant matrices \(\mathcal{Z}_N\) and \(\mathcal{Z}_S\).
To do this we find the discrete Fourier transform of the cyclic quantum dilogarithm.
Similar results are given by \citeauthor{Hikami2014} \cite{Hikami2014}.

\begin{theorem}
  \label{thm:qlf-fourier}
  The Fourier transform of the cyclic quantum dilogarithm is given by
  \begin{equation}
    \label{eq:qlf-fourier}
    \sum_{k=0}^{\nr-1}
    \qlf{\zeta^0, \zeta^1}{k} \omega^{nk}
    =
    \frac{
      \omega^{(\nr-1) \zeta^0} \nr
    }{
      S_{\nr}
    }
    \qlf{-\zeta^{1}, - \zeta^{0}}{n-1}^{-1}
  \end{equation}
  where \(S_{\nr}\) is a constant (depending only on \(\nr\)) satisfying
  \begin{equation}
    \label{eq:SN-value}
    (S_{\nr})^{\nr} = \df{0}^{\nr} e^{-\pi i/12}
    .
  \end{equation}
\end{theorem}

\begin{proof}
  For \(\zeta^{0}, \zeta^{1}\) satisfying \eqref{eq:flattening-condition-tet}, set
  \begin{equation*}
    S_{\nr}(\zeta^{0}, \zeta^{1})
    \defeq
    \frac{
      \omega^{-(\nr-1) \zeta^{1}}
    }{
      \pf{-\zeta^{1}} 
    }
    \omega^{\zeta^{0} \zeta^{1}}
    \sum_{k = 0}^{\nr -1}
    \frac{
      \omega^{k \zeta^{1}}
    }{
      \pf{\zeta^{0} + k} 
    }
    .
  \end{equation*}
  Using \eqref{eq:D-recurrence} we can re-write \cref{thm:exact-qlf-value} as
  \[
    \pf{\zeta^{0}}^{\nr}
    =
    \exp(\pi i \zeta^{0} \zeta^{1} -\dilr(\zeta^{0}, \zeta^{1}))
    \omega^{-\nr(\nr -1)\zeta^{1}} \df{\zeta^{0} + 1}^{-\nr}
  \]
  We can now apply \cref{eq:f-Nth-power} to compute
  \begin{align*}
    &S_{\nr}(\zeta^{0}, \zeta^{1})^{\nr}
    \\
    &=
    \left[
    \frac{
      \omega^{-(\nr-1) \zeta^{1}}
      }{
      \pf{-\zeta^{1}} 
    }
    \omega^{\zeta^{0} \zeta^{1}}
    \sum_{k = 0}^{\nr -1}
    \frac{
      \omega^{k \zeta^{1}}
      }{
      \pf{\zeta^{0} + k} 
    }
    \right]^{\nr}
    \\
    &=
    \frac{
      \omega^{-\nr(\nr - 1)\zeta^{1}}
    }{
      \pf{\zeta^{0}}^{\nr} 
      \pf{-\zeta^{1}}^{\nr} 
    }
    \omega^{\nr \zeta^{0} \zeta^{1}}
    \left[
      \omega^{(\nr-1)\zeta^{0}}
      \frac{
        \df{0}
      }{
        \df{1 - \zeta^{1}}
        \df{\zeta^{0} + 1}
      }
    \right]^{\nr}
    \\
    &=
    \frac{
      e^{2 \pi i \zeta^{0} \zeta^{1}}
      \df{0}^{\nr}
    }{
      \left[
        \pf{\zeta^{0}}^{\nr}
        \df{\zeta^{0} + 1}^{\nr}
        \omega^{\nr (\nr -1) \zeta^{1}}
      \right]
      \left[
        \pf{-\zeta^{1}}^{\nr}
        \df{-\zeta^{1} + 1}^{\nr}
        \omega^{-\nr (\nr -1) \zeta^{0}}
      \right]
    }
    \\
    &=
    \df{0}^{\nr}
    \exp \leftfun(
      \dilr(\zeta^{0} , \zeta^{1})
      +
      \dilr(-\zeta^{1} , -\zeta^{0})
    \rightfun)
  \end{align*}
  Applying \cref{thm:dilog-inversion-classical} gives \eqref{eq:SN-value}.

  We now know that
  \[
    \pf{\zeta^{0}} 
    =
    \frac{
      \omega^{\zeta^{0} \zeta^{1}}
    }{
      S_{\nr}
    }
    \sum_{k=0}^{\nr-1}
    \frac{
      \omega^{(\nr - 1 -k) \zeta^{0}}
      }{
      \pf{-\zeta^{1} + k} 
    }
  \]
  so that
  \[
    \sum_{\ell = 0}^{\nr-1}
    \pf{\zeta^{0} + \ell}
    \omega^{-\ell \zeta^{1}}
    =
    \frac{
      \omega^{\zeta^{0} \zeta^{1}}
    }{
      S_{\nr}
    }
    \sum_{k,\ell=0}^{\nr-1}
    \frac{
      \omega^{(\nr - 1 -k) (\zeta^{0} + \ell)}
      }{
      \pf{-\zeta^{1} + k} 
    }
  \]
  The expression inside the sum depends only on \(k\) modulo \(\nr \mathbb{Z}\).
  Taking the sum over \(\ell\) gives a periodic delta function which lets us eliminate the sum over \(k\) to get
  \begin{equation}
    \sum_{k=0}^{\nr-1}
    \pf{\zeta^{0} + \ell}
    \omega^{-\ell \zeta^{1}}
    =
    \frac{
      \nr
      }{
      S_{\nr}
    }
    \frac{
      \omega^{\zeta^{0} \zeta^{1} + \nr \zeta^{0}}
      }{
      \pf{-\zeta^{1} - 1}
    }
  \end{equation}  
  which is equivalent to \eqref{eq:qlf-fourier}.
\end{proof}

\begin{remark}
  Experimental evidence (following \cite[eq.\ (C.5)]{Kashaev1993}) suggests that
  \[
    S_{\nr}
    =
    \nr^{1/2}
    \exp \leftfun( \frac{\pi i}{12} ( \nr + 1/\nr) - \frac{\pi i}{4} \rightfun)
  \]
  One could probably prove this by taking the \(\zeta^{0} \to 0\) limit in \eqref{eq:qlf-fourier}.
\end{remark}

\begin{theorem}
  \label{thm:hard-determinant}
  For \(\zeta^0, \zeta^1\) satisfying \eqref{eq:flattening-condition-tet}, the linear map \(\CC^\nr \to \CC^\nr\) defined by 
  \[
    \mathcal{Z}(\vbh n)
    =
    \sum_{n'=0}^{\nr -1}
    \qlf{\zeta^0, \zeta^1}{n'-n}\vbh{n'}
  \]
  has determinant
  \[
    \det \mathcal{Z} = 
    \frac{\nr^{\nr}}{\df{0}^{\nr}}
    \omega^{\nr(\nr-1) \zeta^{0}/2}
    \leftfun(
      -
      \frac{
        \dilr(\zeta^{0}, \zeta^{1})
      }{
        2\pi i
      }
    \rightfun)
    .
    \qedhere
  \]
\end{theorem}
\begin{proof}
  Using \cref{eq:basis-and-dual-basis} it is easy to see that \(\mathcal{Z}\) is diagonal in the basis \(\set{\vb n}\),
  \[
    \mathcal{Z}(\vb n)
    =
    \sum_{k=0}^{\nr-1} \omega^{-nk} \qlf{\zeta^0, \zeta^1}{-k} \vb n
  \]
  and \cref{eq:qlf-fourier} shows that the matrix coefficients are
  \[
    \sum_{k=0}^{\nr-1} \omega^{-nk} \qlf{\zeta^0, \zeta^1}{-k}
    =
    \frac{\omega^{(\nr -1)\zeta^0} \nr}{S_{\nr}}
    \qlf{-\zeta^{1}, -\zeta^{0}}{n - 1}^{-1}
  \]
  Now by periodicity
  \[
    \det \mathcal{Z}
    =
    \frac{\omega^{\nr (\nr -1)\zeta^0} \nr^{\nr} }{(S_{\nr})^{\nr}}
    \prod_{n=0}^{\nr -1}
    \qlf{-\zeta^{1}, -\zeta^{0}}{n}^{-1}
  \]
  so by \cref{thm:qlf-product,thm:qlf-fourier} the determinant is
  \[
    \frac{\nr^{\nr}}{\df{0}^{\nr}}
    \omega^{\nr(\nr-1) \zeta^{0}/2}
    \leftfun(
    \frac{
      \pi i
      }{
      12 
    }
    +
        \dilr(-\zeta^{1}, -\zeta^{0})
    \rightfun)
  \]
  Applying \cref{thm:dilog-inversion-classical} gives the claim.
\end{proof}

\printbibliography

\end{document}